\documentclass[11pt]{article}

\usepackage{setspace}
\usepackage[utf8]{inputenc}

\usepackage{amssymb,amsmath,amsthm,amsfonts}
\usepackage[margin=1truein]{geometry}

\usepackage{xcolor}
 \definecolor{bordeaux}{RGB}{100,0,50}
 \definecolor{darkblue}{RGB}{25, 25, 112}

\usepackage[pdftex,colorlinks=true,linkcolor=blue,citecolor=blue,urlcolor=red,unicode=true,hyperfootnotes=false,bookmarksnumbered]{hyperref}
\usepackage[capitalise, compress, nameinlink, noabbrev]{cleveref}
\hypersetup{linkcolor={{blue!70!black}}, citecolor={black}, urlcolor={blue!70!black}}
\hypersetup{linkcolor=bordeaux, citecolor = darkblue, urlcolor = darkblue}

\usepackage{enumerate,color}

\newcommand{\jt}[1]{\textcolor{magenta}{\textbf{[Jane: }#1]}}
\newcommand{\pc}[1]{{\color{blue}{\bf [~Pawel:\ } \color{blue}{\em #1}\color{blue}{\bf~]}}}

\newcommand{\tm}[1]{\textcolor{teal} {{\bf Tamas:} #1}}
\newcommand{\mz}[1]{\textcolor{violet}{\textbf{[Maksim: }#1]}}

\newcommand{\diam}{\operatorname{diam}}

\theoremstyle{plain}
\newtheorem{theorem}{Theorem}
\newtheorem{lemma}[theorem]{Lemma}
\newtheorem{claim}[theorem]{Claim}

\newtheorem{corollary}[theorem]{Corollary}

\theoremstyle{definition} 

\newtheorem{remark}[theorem]{Remark}

\author{Mikhail Isaev\thanks{School of Mathematics and Statistics, UNSW Sydney, Sydney, NSW, 2052, Australia} \and Tam\'{a}s Makai\thanks{Mathematisches Institut der Universit\"{a}t M\"{u}nchen, Munich D-80333, Germany} \and Brendan McKay\thanks{School of Computing, Australian National University, Canberra, ACT, 2601, Australia} \and Pawe{\l} Pra{\l}at\thanks{Department of Mathematics, Toronto Metropolitan University, Toronto, ON M5B 2K3, Canada} \and Jane Tan\thanks{Mathematical Institute, University of Oxford, Oxford OX2 6GG, UK} \and Maksim Zhukovskii\thanks{School of Computer Science, The University of Sheffield, Sheffield S1 4DP, UK}}

\date{}

\title{Canonical labelling of random regular graphs}

\begin{document}

\maketitle 

\begin{abstract}
We prove that whenever $d=d(n)\to\infty$ and $n-d\to\infty$ as $n\to\infty$, then with high probability for any non-trivial initial colouring, the colour refinement algorithm distinguishes all vertices of the random regular graph $\mathcal{G}_{n,d}$. This, in particular, implies that with high probability $\mathcal{G}_{n,d}$ admits a canonical labelling computable in time $O(\min\{n^{\omega},nd^2+nd\log n\})$, where $\omega<2.372$ is the matrix multiplication exponent. 
\end{abstract}

\section{Introduction}

Given an input graph $G$, a {\it canonical labelling} algorithm computes a bijection $\pi_G:V(G)\mapsto\{1,\ldots,n\}$ with the following property: if a graph $G'$ is isomorphic to $G$, then the relabelled versions of $G$ and $G'$, under the actions of $\pi_G$ and $\pi_{G'}$, are identical. There is a linear time reduction from the graph isomorphism problem to canonical labelling: once the labellings $\pi_G,\pi_{G'}$ have been computed for two input graphs $G,G'$, it takes time $O(|E(G)|)$ to check whether $G,G'$ are isomorphic. The best known (in the worst case) algorithm for the graph isomorphism problem is due to Babai~\cite{Babai_iso,Helfgott}: it runs in time $\exp(O(\log^3 n))$ for $n$-vertex graphs. A quasi-polynomial bound $\exp(\log^{O(1)}n)$ is also known for the canonical labelling problem~\cite{Babai_canonical}. Nevertheless, for graphs with bounded degree, both problems can be solved in polynomial time~\cite{Luks,Luks-iso}. In particular, this is the case for $d$-regular graphs when $d=\mathrm{const}.$ In this paper, we show that there exists a polynomial-time canonical labelling algorithm for {\it almost all} $d$-regular graphs for {\it all $0\leq d\leq n-1$}. 

Colour refinement (CR) is a simple algorithmic routine that operates on vertex-coloured graphs. For an input graph $G$ with initial colouring $C_0:V(G)\to\mathbb{Z}$, CR iteratively computes new colourings. At round $t$, $C_t(v)$ is a pair $(C_{t-1}(v),C_{t-1}(N(v)))$, where $C_{t-1}(N(v))$ is the multiset of $C_{t-1}$-colours of neighbours of $v$. That is, the process refines the initial partition $C_0$ and halts once the partition stabilises. Let us call a colouring {\it discrete} if every pair of vertices is coloured differently. If CR runs on an uncoloured graph (i.e., there is only one initial colour) and outputs a discrete colouring, then, since the vertex colours are isomorphism-invariant, this yields a canonical labelling by numbering the colour names in the lexicographic order. In~\cite{BES,CzajkaPan,GRS}, it was proved that CR run on a binomial random graph $\mathcal{G}(n,p)$\footnote{The vertex set of $\mathcal{G}(n, p)$ is $\{1,\ldots,n\}$, and each pair of vertices is adjacent with probability $p = p(n)$, independently of the other pairs.} outputs a discrete colouring with high probability ({\it whp}, in what follows)\footnote{A sequence of events $\mathcal{B}_n$ holds with high probability if $\mathbb{P}(\mathcal{B}_n)\to 1$ as $n\to\infty$.} whenever $(1+\varepsilon)\frac{\ln n}{n}<p\leq\frac{1}{2}$, which implies a near linear time algorithm\footnote{There exists an implementation of CR that runs in time $O((n+|E(G)|)\log n)$ on any $n$-vertex graph $G$~\cite{BBG}. However, whp CR halts on $\mathcal{G}(n,p)$ in $O(1)$ rounds, implying a linear time bound $O(n^2p)$.} for canonical labelling of $G(n,p)$~\cite{BBG}. 

We stress that for regular uncoloured graphs $G$, CR terminates immediately with a trivial colouring, and is therefore unsuitable for canonical labelling. For a positive integer $n$, we denote $[n]:=\{1,\ldots,n\}$. Let $d\leq n-1$ be a non-negative integer such that $dn$ is even. Let $\mathcal{G}_{n,d}$ be a uniform distribution over all $d$-regular graphs on $[n]$. We write $\mathbf{G}_n\sim\mathcal{G}_{n,d}$ for a graph sampled from this distribution, i.e.\ $\mathbf{G}_n$ is a uniformly random $d$-regular graph on $[n]$. Since, for constant $d$, efficient canonical labelling algorithms are known, we focus on the case $d=\omega(1)$. Moreover, since the edge complement of a $d$-regular graph is $(n-1-d)$-regular, the edge complement of $\mathcal{G}_{n,n-1-d}$ is distributed as $\mathcal{G}_{n,d}$. Therefore, we may restrict ourselves to $d\leq n/2$. Our main result shows that the triviality of the initial colouring is the only obstacle for a complete refinement: once a non-trivial initial colouring is produced, whp CR run on $\mathbf{G}_n\sim\mathcal{G}_{n,d}$ outputs a discrete colouring which is suitable for canonical labelling.

Before we state the main result of this paper, we need one more definition. For a connected graph $G$, we denote by $\diam (G)$ its diameter.

\begin{theorem}\label{thm:main}
Let $d_0$ be large enough, let $d=d(n)$ be such that $d_0\leq d\leq n/2$, and let $\mathbf{G}_n \sim \mathcal{G}_{n,d}$. Then, the following holds whp:
for every non-trivial partition $[n]=V_1\sqcup V_2$ of the vertex set of $\mathbf{G}_n$, CR runs at most $2\diam (\mathbf{G}_n)+3$ steps on $\mathbf{G}_n$ and outputs a discrete colouring.
\end{theorem}

\begin{remark}
We did not try to optimise the bound on the number of rounds, and we believe that $2\diam (\mathbf{G}_n)+3$ is suboptimal. In particular, for $d\geq n^{1/2+\varepsilon}$, we show that $2\diam (\mathbf{G}_n)+1=5$ rounds is enough. Actually, it is natural to suspect that the total number of rounds needed is $(1+o_{\diam (\mathbf{G}_n)}(1))\diam (\mathbf{G}_n)$ whp. We also note that our technique does not directly generalise to small values of $d$; in particular, there appears to be a natural barrier at $d=3$. Nevertheless, we believe that the statement of \cref{thm:main} can be extended to all $3\leq d\leq n/2$.
\end{remark}

We now describe a possible approach to canonical labelling of $\mathbf{G}_n\sim\mathcal{G}_{n,d}$ based on our main result. Since the output of CR can be computed in time $O((n+|E(G)|)\log n)$ on an $n$-vertex graph $G$~\cite{BBG}, \cref{thm:main} reduces the problem to efficiently finding an isomorphism-preserving partition of $[n]$. To this end, recall that, for any $d=\omega(1)$, whp $\mathbf{G}_n$ contains a triangle~\cite{KSV}. Let $t_i$ be the number of triangles that contain the vertex $i\in[n]$ in $\mathbf{G}_n$, and let $t=\max\{t_1,\ldots,t_n\}.$ Using fast square matrix multiplication, it is possible to compute the vector $(t_1,\ldots,t_n)$ in time $n^{\omega}$, where $\omega<2.372$~\cite{matrix}. Alternatively, this vector can be computed in time $O(nd^2)$ using the standard triangle-listing algorithm~\cite{listing}, which is faster when $d<n^{0.685}$, assuming the best known upper bound on the matrix multiplication exponent $\omega$ (and yields the bound $o(n^2)$ for $d=o(\sqrt{n})$, which is faster than any algorithm based on square matrix multiplication since $\omega\geq 2$). Then, we can partition $[n]$ into sets $V_1,V_2$ where $V_1$ contains all vertices $i$ for which $t_i=t$ and $V_2$ contains the rest. It remains to show that this partition is non-trivial.
 
Since whp $\mathbf{G}_n$ contains a triangle, $V_1$ is non-empty. When $d=o(n^{1/3})$, the number of triangles is sublinear whp (by \cref{cl:H-inclusion} below), which immediately implies that there are vertices that do not belong to a triangle, and so the set $V_2$ is non-empty as well. In order to show that $V_2$ is non-empty (i.e., that there are two vertices $u,v$ with $t_u\neq t_v$) whp when $d=\omega(n^{1/3})$, fix two vertices $u,v$ and expose their neighbourhoods $N(u),N(v)$. Then expose the edges inside $N(u)$. Assuming $t_u=t_v$, the set of exposed edges identifies the number of edges in $E(\mathbf{G}_n[V])\setminus E(\mathbf{G}_n[U])$. By standard counting arguments (or using switchings), it follows directly that, for any fixed  $x=x(n)\in\mathbb{Z}_{\geq 0}$, the probability that the latter set contains exactly $x$ edges tends to zero. Finally, when $d=\Theta(n^{1/3})$, it is known that the number of triangles $X_n$ in $\mathbf{G}_n$ satisfies a central limit theorem: $\frac{X_n-\mathbb{E}X_n}{\sqrt{\mathrm{Var}X_n}}$ converges in distribution to a standard normal random variable as $n\to\infty$~\cite{Gao-triangles}. In particular, if $\mathbb{P}(\forall i\,\, t_i=t)>\varepsilon$, then $\mathbb{P}(X_n/n\in\mathbb{Z})>\varepsilon$, contradicting the central limit theorem. Therefore, we get the following.


\begin{corollary}
Let $d=d(n)$ be such that $d\to\infty$ and $n-d\to\infty$, and let $\mathbf{G}_n\sim\mathcal{G}_{n,d}$. There exists an algorithm that runs in time $O(\min\{n^{\omega},nd^2+nd\log n\})$ on $d$-regular $n$-vertex graphs and whp outputs a canonical labelling of $\mathbf{G}_n$.
\end{corollary}




\paragraph{Related work.} In~\cite{Bol-WL,Kucera}, it is proved that, when $d\leq n^{\varepsilon},$ for a sufficiently small $\varepsilon$, then with high probability $\mathbf{G}_n\sim\mathcal{G}_{n,d}$ admits a canonical labelling via the 2-dimensional Weisfeiler-Leman algorithm (2-WL)~\cite{WL}, which is a generalisation of CR where the colouring is applied to pairs of vertices, and whose running time is $O(n^3 \log n)$~\cite{WL-time}. The weaker version of 2-WL suggested by Boll\'{o}bas~\cite{Bol-WL} canonically labels $\mathbf{G}_n$ in time $O(n^{3/2+\varepsilon})$ whp. Moreover, Ku\v{c}era~\cite{Kucera} claimed that his version of the algorithm runs in average time $O(nd)$ for $d=O(1)$ which we failed to verify\footnote{The paper does not provide a proof of this fact. However, the algorithm computes the vertices on shortest cycles as a subroutine and uses the following assertion~\cite[Theorem 3.1]{Kucera}: All cycles of the length $k$ in $d$-regular graph can be found in time $O(n\min\{n,(d-1)^{k/2}\})$. First, we believe that the factor $(d-1)^{k/2}$ should instead read $(d-1)^{\lceil k/2\rceil}$ (for instance, it is unclear how all triangles could be found in time $nd^{1.5}$ --- say, in a union of $(d+1)$-cliques, there are $\Theta(nd^2)$ triangles, which gives the lower bound $\Omega(nd^2)$ on time needed to list them). Second, this bound (even with the fractional power) is not enough to get the expected time $O(nd)$. Indeed, $\mathbf{G}_n$ has a triangle with asymptotic probability $1-\exp(-(d-1)^3/6)$. So we get the bound on the expected time to be at least $(1-\exp(-(d-1)^3/6)-o_n(1))nd^{1.5}\sim nd^{1.5}$ as $d\to\infty$.}. We therefore conclude this paragraph by asking whether there exists a linear-time algorithm --- running in time $O(nd)$ on $d$-regular $n$-vertex graphs --- that canonically labels $\mathbf{G}_n$ whp, at least for some $3\leq d\leq n/2$ (either for all $d=O(1)$, where the hidden constant factor does not depend on the time complexity, or for some $d=\omega(1)$).

For binomial random graphs $\mathcal{G}(n,p)$, the study of canonical labelling algorithms has been more extensive. Babai, Erd\H{o}s, and Selkow~\cite{BES} proved that CR outputs a canonical labelling of $\mathcal{G}(n,1/2)$ in linear time whp since it performs only a bounded number of refinement steps. The argument of \cite{BES} can be extended to show \cite[Theorem~3.17]{Bollobas_book} that the CR colouring of $\mathcal{G}(n,p)$ is whp discrete for all $n^{-1/5}\ln n\ll p\leq 1/2$. Bollob\'{a}s~\cite{Bol-WL} showed a polynomial time canonical labelling algorithm for $(1+\varepsilon)\frac{\ln n}{n}\leq p\leq 2 n^{-11/12}$, which is a weaker version of 2-WL. The next improvement was obtained by Czajka and Pandurangan~\cite{CzajkaPan}: they extended the range of applicability of CR to $\frac{\ln^4 n}{n\ln\ln n}\ll p\leq \frac{1}{2}$, which was finally extended to $(1+\varepsilon)\frac{\ln n}{n}\leq p\leq\frac{1}{2}$ by Gaudio, R\'{a}cz, and Sridhar \cite{GRS}. Linial and Mosheiff~\cite{LM} showed that 2-WL outputs canonical labelling of $\mathcal{G}(n,p)$ whp when $\frac{1}{n}\ll p\leq\frac{1}{2}$. Finally, a polynomial time algorithm that labels canonically $\mathcal{G}(n,p)$ with high probability for all $0\leq p=p(n)\leq 1$ was independently established in~\cite{AKM,VZ}, with CR as the main ingredient (note that below the connectivity threshold $\frac{\ln n}{n}$, whp $\mathcal{G}(n,p)$ contains many isolated vertices, so CR-colouring is not discrete). 

A partition $[n]=V_1\sqcup\ldots\sqcup V_t$ of the vertex set of a graph $G$ is called {\it equitable} if, for any $1\leq i,j\leq t$, not necessarily different, any two vertices in $V_i$ have exactly the same number of neighbours in $V_j$. Clearly, if a graph $G$ admits a non-trivial equitable partition $V_1\sqcup\ldots\sqcup V_t$, then CR does not refine it. The opposite statement is also true --- if there is a partition that CR does not refine, then this partition is equitable. \cref{thm:main} implies that whp CR refines any non-trivial partition of $G_{n,d}$ with any number of parts. Therefore, it implies that whp $\mathbf{G}_n\sim\mathcal{G}_{n,d}$ does not have an equitable partition other than those with 1 part or $n$ parts. If a graph $G$ has a non-trivial automorphism group $\mathrm{Aut}(G)$, then for any non-trivial automorphism $\sigma\in\mathrm{Aut}(G)$, its cycle decomposition identifies an equitable partition of $G$. Therefore, if a graph does not have a non-trivial equitable partition, then its group of automorphisms is either trivial or cyclic. However, in the second case the graph must be vertex-transitive, which is not the case for $\mathbf{G}_n$ whp. In particular, as noted above, whp there exist two vertices that belong to different numbers of triangles. As a result, \cref{thm:main} implies that $\mathbf{G}_n$ is asymmetric whp for all $d$ such that $d\to\infty$ and $n-d\to\infty$. This is a known result for the entire range $3\leq n\leq d-4$, having been first established for $d=o(\sqrt{n})$ in~\cite{MW-asymmetry} and then for $d\gg\log n$ in~\cite{KSV-asymmetry}. 

\paragraph{Proof strategy.} The key step to obtain \cref{thm:main} is to show that after some number of rounds of CR, one can coarsen the partition associated with colours to get a partition with $k$ parts of comparable size (for any desired arbitrarily large constant~$k$). The following statement makes this precise. Note that we may coarsen a partition at any stage of the CR algorithm if it is convenient for the argument that follows. Clearly, one may couple the original process with the modified one so that the partitions in the original process are refinement of the corresponding partitions in the modified one. In particular, if the modified process reaches discrete colouring, then so does the original one.

\begin{theorem}\label{thm:main_many_large_set}
Let $d_0$ be large enough, let $d_0\leq d=d(n)\leq n/2$, and let $\mathbf{G}_n \sim \mathcal{G}_{n,d}$. Let $k\in\mathbb{N}$ be an arbitrary constant. Then, the following holds whp: for every non-trivial partition $[n]=V_1\sqcup V_2$ of the vertex set of $\mathbf{G}_n$, after $\diam (\mathbf{G}_n)+2$ rounds of CR, there exists a partition $[n] = V_1 \sqcup \ldots \sqcup V_k$ such that for any $i \in [k]$, $n/3k \le |V_i| \le 3n/k$ and $V_i$ is a union of some colour classes (in other words, $V_1 \sqcup \ldots \sqcup V_k$ is a coarsening of the partition associated with resulting colour classes). 
\end{theorem}

Once there are many parts of linear and comparable sizes, CR distinguishes all vertices after some additional number of rounds.

\begin{theorem}\label{thm:main_complete_refinement}
Let $d_0$ be large enough, let $d_0\leq d=d(n) \leq n/2$, and let $\mathbf{G}_n \sim \mathcal{G}_{n,d}$. There exists some universal large constant $k\in\mathbb{N}$ such that the following holds whp: for every initial partition $[n] = V_1 \sqcup \ldots \sqcup V_k$ such that for any $i \in [k]$, $n/3k \le |V_i| \le 3n/k$, CR terminates on $\mathbf{G}_n$ after at most $\diam (\mathbf{G}_n)+1$ rounds and outputs a discrete colouring.
\end{theorem}

\cref{thm:main} follows immediately from \cref{thm:main_many_large_set} and \cref{thm:main_complete_refinement}. 
To prove each of these two theorems, we address two regimes for $d$ separately: the dense case with $d\geq n^{1/2+\varepsilon}$, and the sparse case with $d_0\leq d\leq n^{10/17}$. The corresponding statements are reiterated in \cref{sc:dense-many-large,,sc:sparse-many-large,,sc:dense-full,,sc:sparse-full}.

\cref{thm:main_many_large_set} for $d\geq n^{1/2+\varepsilon}$ is proved in \cref{sc:dense-many-large}. Its proof consists of two parts. First, we show that whp after one refinement round there exists a coarsening $[n]=U_1\sqcup U_2$ of the CR-partition such that $|U_1|,|U_2|\gg n/d$ (\cref{cl:intermediate_small}). One more round is needed to get all colour classes of size at most $\delta n$, for an arbitrary constant $\delta>0$ (\cref{cl:intermediate_large}). The latter claim follows from the fact that there is no large set with all vertices having the same degree profile with respect to $(U_1,U_2)$. This is the main technical complication in the proof of \cref{sc:dense-many-large} in the dense case: although this fact is easy to show in $\mathcal{G}(n,p)$, in random regular graphs we cannot rely on local limit theorems. Instead we use asymptotic estimations of the number of graphs with a given degree sequence as well as anti-concentration properties of the hypergeometric distribution. The sparse case $d=o(n)$ is addressed in \cref{sc:sparse-many-large}. Here, we show that, for every initial colouring $V_1\sqcup V_2$, where $|V_1|<cn$, for a sufficiently small constant $c>0$, after a few rounds of colour refinement, we will get a union of colour classes $U$ of size $|U|\in[cn/d,cn]$ (\cref{lm:common_expansion_Thm3}). One more round is needed to get a set $U'$ of size $|U'|\in[n\ell/d,0.999n]$, for an arbitrarily large constant $\ell$~(\cref{lem:exactly_s_neighbours}). Then, similarly to the dense case, we show that there is no set of size more than $\delta n$ such that all its vertices have same number of neighbours in $U'$, that gives us the desired partition (\cref{lm:from_single_set-to_partition-gen}). All three lemmas rely on switching arguments. In particular, the last two lemmas use switchings to establish an analogue of the Erd\H{o}s--Littlewood--Offord theorem in the context of uniformly random graphs with a fixed degree sequence.

\cref{thm:main_complete_refinement} is proved in \cref{sc:dense-full,,sc:sparse-full}. The dense case is significantly easier. For instance, when $d=\Theta(n)$, two refinement rounds are enough to obtain a discrete colouring whp. Indeed, let $u,v$ be two fixed vertices. Expose the neighbourhoods $N(u),N(v)$, and all the edges that touch $N(v)$. The exposed edges identify degree profiles of vertices in $N(u)\setminus N(v)$ with respect to the fixed partition. Since the latter set has size $\Theta(n)$, it is extremely unlikely that all the degrees are equal to the fixed values. Clearly, the probability of this event is $\left(1/\sqrt{n/k}\right)^{\Theta(kn)}$ in $\mathcal{G}(n,p)$, which is enough to overcome the union bound with room to spare. In order to transfer this bound to random regular graphs, we use asymptotic enumeration of graphs with a given degree sequence and anti-concentration inequalities for hypergeometric distribution, as for the dense case in \cref{thm:main_many_large_set}. For $n^{1/2+\varepsilon}\leq d=o(n)$, we need one additional refinement round in order to reach a set of vertices at distance at most 2 from $\{u,v\}$ of size $\Theta(n)$. The sparse case $d\leq n^{10/17}$, addressed in \cref{sc:sparse-full}, requires a more delicate switching argument and constitutes the most technical part of the paper. Here, in order to reach a set of size $\Theta(n)$, from fixed vertices $u,v$, we need $\mathrm{diam}(\mathbf{G}_n)$ rounds. Then, in contrast to the sparse case in \cref{thm:main_many_large_set}, we need a multidimensional analogue of the Erd\H{o}s--Littlewood--Offord theorem (\cref{cl:large-switchings-2}), since the degree profiles are considered with respect to $k$ sets of the partition. Nevertheless, the claim can still be established by applying a similar switching argument $\Theta(k)$ times.  


\paragraph{Organisation.} We start by presenting some preliminary results on properties of random graphs and concentration inequalities in \cref{sc:pre}. The rest of the paper is devoted to the proof of \cref{thm:main_many_large_set,thm:main_complete_refinement} which immediately imply our main \cref{thm:main}. \cref{thm:main_many_large_set} is proved across \cref{sc:dense-many-large,sc:sparse-many-large} where the dense and sparse cases are treated respectively. \cref{sc:dense-full,sc:sparse-full} are devoted to the dense and sparse cases of \cref{thm:main_complete_refinement}. Finally, in \cref{sc:expansion-lemma-proof} we present a proof of the expansion \cref{lm:common_expansion_Thm3} that asserts that whp for every set, its size remains concentrated after multiple expansion rounds. This lemma may be of independent interest and useful in other contexts.

\paragraph{Notation.} For a graph $G$, a set of vertices $U\subseteq V(G)$, and a non-negative integer $r$, we denote by $S_r(U)$ the sphere of radius $r$ around $U$ in the graph metric, omitting the dependency on $G$ since the underlying graph is always clear from the context. That is, $S_r(U)$ consists of vertices $v$ such that the length of a shortest path from $v$ to $U$ equals $r$. In particular, $S_0(U)=U$. We also denote $B_r(U)=\cup_{0\leq i\leq r}S_i(r)$ the ball of radius $r$ around $U$. We sometimes denote $S_1(U)$ by $N(U)$ and refer to it as the neighbourhood of $U$. For a set of vertices $X$ and a vertex $x\notin X$, we denote by $N_X(x)$ the number of neighbours of $x$ in $X$. We also use the standard notation $G[U]$ for the subgraph of $G$ induced by a set $U\subseteq V(G)$, and $G[U\times V]$ for the bipartite subgraph with (disjoint) parts $U$ and $V$, consisting of all edges of $G$ with one endpoint in $U$ and the other in $V$.

For a given degree sequence $\mathbf{d}=(d_1,\ldots,d_n)$, we will use $g(\mathbf{d})$ to denote the number of graphs on the vertex set $[n]$ with the degree sequence $\mathbf{d}$.

We often write $A\sqcup B$ to denote the union of two {\it disjoint} sets $A$ and $B$.

Finally, for a random variable $X$ with distribution $Q$, we write $X\sim Q$. In particular, $X\sim\mathrm{Bin}(n,p)$ is a binomial random variable with $n$ trials and success probability $p$.

\section{Preliminaries} 
\label{sc:pre}
In this section, we collect some probabilistic tools as well as properties of random regular graphs that we will use in the main proofs to follow.
\subsection{Concentration Inequalities} 

We will use the following specific instances of Chernoff's bound --- see, for example,~\cite[Theorem~2.1]{Janson}. Let $X \sim \textrm{Bin}(n,p)$. 
 Then, 
  for any $t \ge 0$ we have
	\begin{eqnarray}
		\mathbb{P} ( X \ge \mathbb{E} X + t ) &\le& \exp \left( - \frac {t^2}{2 (\mathbb{E} X + t/3)} \right)  \label{eq:chern_lb} \\
		\mathbb{P} ( X \le \mathbb{E} X - t ) &\le& \exp \left( - \frac {t^2}{2 \mathbb{E} X} \right).\label{eq:chern_ub}
	\end{eqnarray}
\begin{remark}
\label{rk:hypergeometric}
The same bounds hold for a random variable with the hypergeometric distribution with parameters $N$, $n$, and $m$ (see, for example,~\cite[Theorem 2.10]{Janson}). 
\end{remark}

%

\subsection{Properties of Binomials} 

Now, let us start with a few auxiliary observations. 

\begin{lemma}\label{lem:aux}
For all positive integers $b_1 < a_1$, $b_2 < a_2$, 
$$
\frac{a^a}{b^b(a-b)^{a-b}}\geq\frac{a_1^{a_1}a_2^{a_2}}{b_1^{b_1}(a_1-b_1)^{a_1-b_1}b_2^{b_2}(a_2-b_2)^{a_2-b_2}},
$$
where $a=a_1+a_2$ and $b=b_1+b_2$.
\end{lemma}

\begin{proof}
Let us fix non-negative integers $b_1 < a_1$, $b_2 < a_2$. Let $a=a_1+a_2$ and $b=b_1+b_2$. We start with the following inequality, which follows immediately from the Vandermonde's identity: for any $n \in \mathbb{N}$,
\begin{equation}\label{eq:Vandermonde}
   \binom{na}{nb}^{1/n} = \left( \sum_{k=0}^{nb} \binom{na_1}{k} \binom{na_2}{nb-k} \right)^{1/n} \ge \binom{na_1}{nb_1}^{1/n} \binom{na_2}{nb_2}^{1/n}.
\end{equation}
Using Stirling's formula ($s! = (1+o(1)) \sqrt{2\pi s} (s/e)^s$), the left hand side of \eqref{eq:Vandermonde} can be estimated as follows:
\begin{eqnarray*}
   \binom{na}{nb}^{1/n} &=& \biggl( \frac { (na)! }{ (nb)! \cdot (n(a-b))! } \biggr)^{1/n} \\
   &=& \biggl( (1+o(1)) \frac { \sqrt{2\pi na} (na/e)^{na} }{ \sqrt{2\pi nb} (nb/e)^{nb} \cdot \sqrt{2\pi n(a-b)} (n(a-b)/e)^{n(a-b)} } \biggr)^{1/n} \\
   &=& \frac {a^a}{b^b (a-b)^{a-b} } \left( (1+o(1)) \frac {a}{2 \pi n b (a-b)} \right)^{1/2n} \\
   &=& \frac {a^a}{b^b (a-b)^{a-b} } \exp \left( - \Theta \left( \frac {\log n}{n} \right) \right) \\
   &\to& \frac {a^a}{b^b (a-b)^{a-b} },
\end{eqnarray*}
as $n \to \infty$. The right hand side of \eqref{eq:Vandermonde} can be dealt with the same way:
$$
\binom{na_1}{nb_1}^{1/n} \binom{na_2}{nb_2}^{1/n} \to \frac {a_1^{a_1}}{b_1^{b_1} (a_1-b_1)^{a_1-b_1} } \cdot \frac {a_2^{a_2}}{b_2^{b_2} (a_2-b_2)^{a_2-b_2} },
$$
as $n \to \infty$. This verifies the desired inequality.
\end{proof}

The preceding lemma has a useful corollary, which we record as follows.



\begin{corollary}[Anti-concentration of hypergeometric distribution]
\label{cor:binomial_product}
For integers $0 < b_1 < a_1$, $0 < b_2 < a_2$, 
\begin{equation}\label{eq:hyp0}
\binom{a_1}{b_1}\binom{a_2}{b_2} \le \binom{a}{b},
\end{equation}
and
\begin{equation}\label{eq:hyp1}
\binom{a_1}{b_1}\binom{a_2}{b_2}\leq\frac{2}{3}\sqrt{\frac{b(a-b) a_1a_2}{a b_1(a_1-b_1)b_2(a_2-b_2)}}\binom{a}{b}, 
\end{equation}
where $a=a_1+a_2$ and $b=b_1+b_2$. In particular, for all integers $0 < b < a$ and positive integers $k$,
\begin{equation}\label{eq:hyp2}
 \binom{a}{b}^k \le \left( \frac {a}{b(a-b)} \right)^{(k-1)/2} \binom{ka}{kb}.
\end{equation}
Moreover, for all integers $0 < b < a$,
\begin{equation}\label{eq:hyp3}
 \binom{2a}{2b} \le 4 \sqrt{ \frac {b(a-b)}{a} } \binom{a}{b}^2 \le 2 \sqrt{a} \binom{a}{b}^2.
\end{equation}
\end{corollary}

\begin{proof}
Let us fix positive integers $b_1 < a_1$, $b_2 < a_2$. Let $a=a_1+a_2$ and $b=b_1+b_2$.
The first inequality is a trivial consequence of Vandermonde's inequality.
We will use the following variant of the Stirling's formula that holds for all positive integers $s$:
$$
    \sqrt{2 \pi s} \left( \frac {s}{e} \right)^s < \sqrt{2 \pi s} \left( \frac {s}{e} \right)^s \exp \left( \frac {1}{12s+1} \right) < s! < \sqrt{2 \pi s} \left( \frac {s}{e} \right)^s \exp \left( \frac {1}{12s} \right) \le e^{1/12} \sqrt{2 \pi s} \left( \frac {s}{e} \right)^s.
$$
Using these inequalities and \cref{lem:aux}, we get
\begin{eqnarray*}
\binom{a_1}{b_1} \binom{a_2}{b_2} &\le& \frac {e^{1/12} \sqrt{2 \pi a_1} (a_1/e)^{a_1}} { \sqrt{2 \pi b_1} (b_1/e)^{b_1} \sqrt{2 \pi (a_1-b_1)} ((a_1-b_1)/e)^{a_1-b_1} } \\
&& \qquad \cdot \ \frac {e^{1/12} \sqrt{2 \pi a_2} (a_2/e)^{a_2}} { \sqrt{2 \pi b_2} (b_2/e)^{b_2} \sqrt{2 \pi (a_2-b_2)} ((a_2-b_2)/e)^{a_2-b_2} } \\
&=& \frac {e^{1/6}}{2\pi} \cdot \sqrt{ \frac {a_1 a_2}{b_1(a_1-b_1)b_2(a_2-b_2)}} \cdot \frac{a_1^{a_1}a_2^{a_2}}{b_1^{b_1}(a_1-b_1)^{a_1-b_1}b_2^{b_2}(a_2-b_2)^{a_2-b_2}} \\
&\le& \frac {e^{1/6}}{2\pi} \cdot \sqrt{ \frac {a_1 a_2}{b_1(a_1-b_1)b_2(a_2-b_2)}} \cdot \frac{a^a}{b^b(a-b)^{a-b}} \\
&=& \frac {e^{1/3}}{\sqrt{2 \pi}} \cdot \sqrt{ \frac {b(a-b)a_1 a_2}{ab_1(a_1-b_1)b_2(a_2-b_2)}} \cdot \frac{\sqrt{2\pi a} (a/e)^a}{e^{1/6} \sqrt{2\pi b}(b/e)^b \sqrt{2\pi (a-b)} ((a-b)/e)^{a-b}} \\
&\le& \frac{2}{3}\sqrt{\frac{b(a-b) a_1a_2}{a b_1(a_1-b_1)b_2(a_2-b_2)}}\binom{a}{b},
\end{eqnarray*}
since $e^{1/3}/\sqrt{2 \pi} \approx 0.557 < 2/3$, proving~\eqref{eq:hyp1}. 

To see that~\eqref{eq:hyp2} follows, note that for any non-negative integers $a<b$ and positive integer $i$,
$$
\binom{a}{b} \binom{ia}{ib} \le \frac {2}{3} \left( \frac {(i+1)a}{bi(a-b)} \right)^{1/2} \binom{(i+1)a}{(i+1)b} \le \left( \frac {a}{b(a-b)} \right)^{1/2} \binom{(i+1)a}{(i+1)b},
$$
since $\frac {2}{3} \sqrt{(i+1)/i} \le 2 \sqrt{2}/3 \approx 0.943 < 1$. Applying the above inequality $k-1$ times, we get the desired conclusion.

Finally, for~\eqref{eq:hyp3}, note that
\begin{eqnarray*}
\binom{2a}{2b} &\le& \frac {e^{1/12} \sqrt{2 \pi (2a)} (2a/e)^{2a}} { \sqrt{2 \pi (2b)} (2b/e)^{2b} \sqrt{2 \pi (2(a-b))} (2(a-b)/e)^{2(a-b)} } \\
&=& e^{5/12} \sqrt{\pi} \cdot \sqrt{ \frac {b(a-b)}{a}} \cdot \left( \frac { \sqrt{2 \pi a} (a/e)^{a}} { e^{1/6} \sqrt{2 \pi b} (b/e)^{b} \sqrt{2 \pi (a-b)} ((a-b)/e)^{a-b} } \right)^2 \\
&\le& 4 \sqrt{ \frac {b(a-b)}{a}} \cdot \binom{a}{b}^2 ~~\le~~ 2 \sqrt{ a } \cdot \binom{a}{b}^2,
\end{eqnarray*}
since $e^{5/12} \sqrt{\pi} \approx 2.689 < 4$, which completes the proof of the corollary.
\end{proof}

\subsection{Counting Graphs} 

Recall that, we denote the number of graphs on the vertex set $[n]$ with a given degree sequence $\mathbf{d}$ by $g(\mathbf{d})$. The following result gives tight asymptotic bounds on $g(\mathbf{d})$ for any degree sequence satisfying some mild condition. Dense graphs were investigated in~\cite{MW} but the result was generalised to sparser graphs in~\cite{LW}. The result below states the bound in a slightly different form, which is more convenient for our purposes. We provide a short proof of this reformulation.


\begin{theorem}[\cite{LW,MW,MW2}, rephrased]
\label{thm:newcount}
Let $\mathbf{d}=(d_1,\ldots,d_n)$ be any degree sequence such that $\sum_{i=1}^n d_i$ is
even and for all $i \in [n]$, $|d_i-d|\le d^{1/2+\varepsilon'}$
for some $\varepsilon'>0$, where $d$ is the average degree.
Let $m=\frac{1}{2}\sum_{i=1}^nd_i=\frac{dn}{2}$ be the number of edges, and
$\eta=\frac1n\sum_{i=1}^n (d_i-d)^2$.
Suppose that  $1\ll d \le (1-\varepsilon_0)n$ for some $\varepsilon_0>0$. 
Then,
$$
  g(\mathbf{d}) = \frac{\prod_{i=1}^n\binom{n-1}{d_i}}
  {m^{1/2} \binom{\binom{n}{2}}{m}} \exp\bigl( O(1)-\Theta(\eta^2/d^2)\bigr).
$$
\end{theorem}
\begin{proof}
  Define $\lambda=d/(n-1)$, $N=\binom n2$, and note that $m=\lambda N$.
  From the overlapping theorems of \cite{LW,MW,MW2}, 
  we know that
   \begin{equation}
  \label{eq:gprecise}
      g(\mathbf{d}) \sim \sqrt 2
        \exp\biggl(\frac14-\frac{\eta^2}{4\lambda^2(1-\lambda)^2 n^2}\biggr)
         \bigl(\lambda^\lambda(1-\lambda)^{1-\lambda}\bigr)^N
         \prod_{i=1}^n\binom{n-1}{d_i}.
  \end{equation}
  By Stirling's formula,
  $$
     \binom{N}{\lambda N} \sim 
     \bigl(\lambda^\lambda(1-\lambda)^{1-\lambda}\bigr)^{-N}
     \bigl(2\pi\lambda(1-\lambda)N\bigr)^{-1/2},
  $$
  provided $\lambda(1-\lambda)N\to\infty$.
  The theorem now follows from noting that
  $\lambda(1-\lambda) n=\Theta(d)$ and
  $\lambda(1-\lambda) N=\Theta(m)$ if $\lambda$ is bounded
  away from~1.
\end{proof}

We say that a sequence $\mathbf{d}=(d_1,\ldots,d_n)$ is \emph{balanced} if $|d_i - d_j| \le 1$ for any $1 \le i < j \le n$. The next observation is that $g(\mathbf{d})$ is maximized (over all sequences with a fixed even sum) when $\mathbf{d}$ is balanced. 

\begin{lemma}
\label{lm:even}
Fix $m\in\left[\binom{n}{2}\right]$. The number of graphs on $n$ vertices and $m$ edges with a specified degree sequence $\mathbf{d}=(d_1,\ldots,d_n)$ (in particular, $\sum d_i=2m$) is maximized when the degree sequence is as even as possible. In other words,
$$
\max \left\{ g(\mathbf{d}) : \sum d_i = 2m \right\} = g(\mathbf{ \hat{d} }), 
$$
where $\mathbf{\hat{d}}$ is the degree sequence, unique up to order, with only $\lfloor \sum d_i / n \rfloor$ and $\lceil \sum d_i / n \rceil$.
\end{lemma}

\begin{proof}
Let us fix $m\in\left[\binom{n}{2}\right]$. For a contradiction, suppose that $g(\mathbf{d})$, the number of graphs on $n$ vertices with the degree sequence $\mathbf{d}=(d_1,\ldots,d_n)$ satisfying $\sum d_i = 2m$, is maximized for a non-balanced sequence $\mathbf{d}$. Without loss of generality, suppose that $d_1\leq d_2-2$. 

Let $\mathbf{d}'=(d'_1, \ldots, d'_n) = (d_1+1,d_2-1,d_3,\ldots,d_n)$ with $\sum d'_i = \sum d_i = 2m$. Consider the following switching: in a graph with the degree sequence $\mathbf{d}$, choose a vertex $v$ such that $1\nsim v$ and $2\sim v$, remove edge $\{2,v\}$ and add $\{1,v\}$ instead (``switch the edges $\{2,v\}$ and $\{1,v\}$'') to get a graph with the degree sequence $\mathbf{d}'$. Letting $g_x(\mathbf{d})$ be the number of graphs where vertices $1$ and $2$ have exactly $x$ common neighbours, we get
$$
g_x(\mathbf{d}) (d_2-x) = g_x(\mathbf{d}') (d_1 + 1 - x) <  g_x(\mathbf{d}') (d_2-x),
$$
since $d_1\leq d_2-2$. Indeed, there are $d_2-x$ graphs with the degree sequence $\mathbf{d'}$ that can be obtained from a given graph with the degree sequence $\mathbf{d}$ and each graph with the degree sequence $\mathbf{d'}$ can be obtained from $d_1 + 1 - x$ graphs with the degree sequence $\mathbf{d}$. It follows that $g(\mathbf{d}')=\sum_xg_x(\mathbf{d}')>\sum_xg_x(\mathbf{d})=g(\mathbf{d})$, giving us the desired contradiction.
\end{proof}

We also recall the probability bound on the event that a uniformly random graph with a given degree sequence contains a specified set of edges.

\begin{claim}[\cite{McKay-inclusion}]
Let $\mathbf{G}_n$ be a uniformly random graph on the vertex set $[n]$ with a fixed degree sequence $(d_1,\ldots,d_n)$. Let $H$ be a graph on the vertex set $[n]$ with degree sequence $(d'_1, \ldots, d'_n)$ such that $d'_i \le d_i$ for all $i \in [n]$. Let $m=\frac{1}{2}\sum d_i$ and $m'=\frac{1}{2}\sum d'_i$ be the number of edges in $\mathbf{G}_n$ and $H$, respectively. Let $d=\max\{d_1,\ldots,d_n\}=o(m^{1/2})$ and let $m'\leq m/2$. Then,
$$
 \mathbb{P}(H\subseteq \mathbf{G}_n)\leq\frac{\prod_{i=1}^n d_i(d_i-1)\ldots(d_i-d'_i+1)}
{2^{m'}(m-2d^2)(m-2d^2-1)\ldots(m-2d^2-m'+1)}.
$$
\label{cl:H-inclusion}
\end{claim}

We note that \cref{cl:H-inclusion} 
 immediately implies the following.


\begin{claim}
\label{cl:regular-probability-bound}
Under the assumptions of \cref{cl:H-inclusion}, for all $n$ large enough,
$$
 \mathbb{P}(H\subseteq\mathbf{G}_n)\leq\left(\frac{2d}{n}\right)^{|E(H)|}.
$$
\end{claim}

\subsection{Sandwiching Graphs} 

Consider the binomial random graph $\mathcal{G}(n,p)$ which has vertex set $[n]$ and each potential edge is included independently at random with probability $p$; $p=p(n)$ could be, and usually is, a function of $n$ that tends to zero as $n \to \infty$. Since the independence of the edges allows the use of a wide variety of techniques, $\mathcal{G}(n,p)$ is typically much easier to study compared to $\mathcal{G}_{n,d}$. As a result, it is tempting to hope for a general purpose ``black box'' theorem that is able to translate results between $\mathcal{G}(n,p)$ and $\mathcal{G}_{n,d}$. In 2004, Kim and Vu~\cite{KV2004} formalized this desire in their famous ``sandwich conjecture''. After more than 20 years and a number of important contributions~\cite{DFRS2017,GIM2020,GIM2022,GIM2023,KRRS2023}, the conjecture was finally proved~\cite{BIM}.

\begin{theorem}[Theorem~1.1 \cite{BIM}]\label{thm:sandwich}
For each $\epsilon>0$ there is some $C>0$ such that the following holds for each $d \ge C \log n$. There is a coupling $(\mathbf{G}_*, \mathbf{G}, \mathbf{G}^*)$ of random graphs such that $\mathbf{G}_* \sim \mathcal{G}(n, (1-\epsilon)d/n)$, $\mathbf{G} \sim \mathcal{G}_{n,d}$, $\mathbf{G}^* \sim \mathcal{G}(n, (1+\epsilon)d/n)$, and whp $\mathbf{G}_* \subset \mathbf{G} \subset \mathbf{G}^*$.
\end{theorem}

\subsection{Expansion Properties} 

We will use the expansion properties of random $d$-regular graphs that follow from their eigenvalues. The adjacency matrix $A=A(G)$
of a given $d$-regular graph $G$ on $n$ vertices, is an $n \times n$ real symmetric matrix. Thus, the matrix $A$ has $n$ real eigenvalues which we denote by $d = \lambda_1 \ge \lambda_2 \ge \cdots \ge \lambda_n$. It is known that several structural properties of a $d$-regular graph are reflected in its spectrum. Since we focus on expansion properties, we are particularly interested in the following quantity: 
$$\lambda=\lambda(G):=\max\{|\lambda_2|,|\lambda_n|\}.$$

The number of edges $e(A,B)$ between two sets $A$ and $B$ in a random $d$-regular graph on $n$ vertices is expected to be close to $d|A||B|/n$. (Note that $A \cap B$ does not have to be empty; in general, $e(A,B)$ is defined to be the number of edges between $A \setminus B$ to $B$ plus twice the number of edges that contain only vertices of $A \cap B$.) A small $\lambda$ (that is, a large spectral gap) implies that the deviation is small. The following bound is very convenient. 

\begin{theorem}[Expander Mixing Lemma~\cite{AlonChung,KS-pseudo}]
Let $G$ be a $d$-regular graph. Then for any two sets of vertices $A,B\subseteq V(G)$, the number $e(A,B)$ of edges of $G$ with one endpoint in $A$ and another endpoint in $B$ satisfies
$$
 \left|e(A,B)-\frac{d|A||B|}{n}\right|\leq\lambda\sqrt{|A||B|}.
$$
\end{theorem}


We will apply the Expander Mixing Lemma together with an asymptotic bound on $\lambda$ for random regular graphs. It was first established by Friedman~\cite{Friedman} for constant $d\geq 3$, confirming the conjecture of Alon~\cite{Alon}. The case of $d\to\infty$ was then conjectured by Vu~\cite{Vu}. After a series of important contributions~\cite{AKV,BFSU,CGJ,KSVW,TY}, it was resolved for all $d=o(n)$ by Bauerschmidt, Huang, Knowles, and Yau~\cite{BHKY} and Sarid~\cite{Sarid}, and then for $d=\Theta(n)$ by He~\cite{He}. 

\begin{theorem}[\cite{BHKY,He,Sarid}]
\label{th:lambda}
Let $3\leq d\leq n/2$ and $\mathbf{G}_n\sim\mathcal{G}_{n,d}$. 
 Then, whp $\lambda(\mathbf{G}_n)\leq(2+o(1))\sqrt{d(1-d/n)}$. 
\end{theorem}

We will also require a finer expansion result ensuring that, for every set, its size remains concentrated after several rounds of expansion.

\begin{lemma}
\label{lm:common_expansion_Thm3}
Let $c>0$ be small enough and $d_0 \in \mathbb{N}$ be large enough (independent of $c$). Let $d_0\leq d \le n/2$ and $\mathbf{G}_n\sim\mathcal{G}_{n,d}$. Then, the following holds whp: 
for every set $U$ of size $u = |U| \leq \frac{cn}{d}$ and for every positive integer $r$ such that $ud(d-1)^{r-1} \le cn$, 
$$
|S_r(U)| \ge (1-100c-4\ln d/d)ud(d-1)^{r-1}.
$$
\end{lemma}

The proof of this lemma is technically involved. To avoid interrupting the flow of the paper, we present it separately in \cref{sc:expansion-lemma-proof}.

\section{Proof of \cref{thm:main_many_large_set}: Dense Case}
\label{sc:dense-many-large}

Here we prove the following.


\begin{theorem}\label{thm:dense_two_rounds}
Let $\varepsilon\in(0,1/2)$, $n^{1/2+\varepsilon} \leq d=d(n)\leq n/2$, and let $\mathbf{G}_n \sim \mathcal{G}_{n,d}$. Let $k\in\mathbb{N}$ be an arbitrary constant. Then, the following holds whp: for every non-trivial partition $[n]=V_1\sqcup V_2$ of the vertex set of $\mathbf{G}_n$, after two rounds of CR, there exists a partition $[n] = V_1 \sqcup \ldots \sqcup V_k$ such that for any $i \in [k]$, $n/3k \le |V_i| \le 3n/k$ and $V_i$ is a union of some colour classes. 
\end{theorem}

\cref{thm:dense_two_rounds} follows easily from the following two claims. 

\begin{claim}\label{cl:intermediate_small}
Let $\varepsilon\in(0,1/2)$, $n^{1/2+\varepsilon} \leq d=d(n)\leq n/2$, and let $\mathbf{G}_n \sim \mathcal{G}_{n,d}$. 
For every $C>1$, the following property holds whp: for every non-trivial partition $[n]=V_1\sqcup V_2$ with $\min\{|V_1|,|V_2|\} < C n/d$, after one round of CR, there exists a partition $[n]=U_1 \sqcup U_2$ with $\min\{|U_1|,|U_2|\} \ge C n/d$ such that $U_1$ and $U_2$ are unions of some colour classes (in other words, $U_1 \sqcup U_2$ is a coarsening of the partition
associated with resulting colour classes).
\end{claim}

\begin{claim}\label{cl:intermediate_large}
Let $\varepsilon\in(0,1/2)$, $n^{1/2+\varepsilon} \leq d=d(n)\leq n/2$, and let $\mathbf{G}_n \sim \mathcal{G}_{n,d}$. 
For any $\delta \in (0,1]$, there exists $C= C(\delta)>0$ such that the following property holds whp: for every non-trivial partition $[n]=V_1\sqcup V_2$ with $\min\{|V_1|,|V_2|\} \ge C n/d$, there is no colour class of size more than $\delta n$ after one round of CR.
\end{claim}

Before we prove these two claims, let us show how they imply \cref{thm:dense_two_rounds}.

\begin{proof}[Proof of \cref{thm:dense_two_rounds}]
Since we aim for the statement that holds whp, we may assume that the statements in \cref{cl:intermediate_small} and in \cref{cl:intermediate_large} hold deterministically. Fix any $k \in \mathbb{N}$, and let $\delta = 1/(3k)$. Let $C = C(\delta)$ be the large enough constant implied by \cref{cl:intermediate_large}. 

Consider any non-trivial partition $[n]=V_1\sqcup V_2$. If $\min\{|V_1|,|V_2|\}<C n/d$, then after one round of CR (and coarsening), we get a partition into two colour classes in which one of the colour classes has size at least $C n/d$ but at most $n/2$ (by~\cref{cl:intermediate_small}). After another round of CR, all colour classes have size at most $\delta n = n /(3k)$ (by~\cref{cl:intermediate_large}). If $\min\{|V_1|,|V_2|\} \ge C n/d$, then we get the above property after a single round of CR. 

To get the desired partition into $k$ parts, each of size at least $n/(3k)$, one can iteratively merge any two colour classes of size at most $\delta n$ until there is at most one such class remaining. After possibly merging this last class (if it exists) with any other arbitrarily chosen class, we get at least $k=1/(3\delta)$ classes (but at most $3k$ of them), each of size at least $\delta n = n/3k$ but at most $3 \delta n$. Finally, if there are more than $k$ classes, one can arbitrarily merge some triples of them (and, perhaps, one pair) to get exactly $k$ classes, each of size at most $9 \delta n = 3n/k$. This finishes the proof of the theorem.
\end{proof}

\cref{cl:intermediate_small} has a short and easy proof so let us start with it. 

\begin{proof}[Proof of \cref{cl:intermediate_small}]
Consider any partition $[n]=V \sqcup ([n] \setminus V)$ with $1 \le |V| < Cn/d$. After one round of CR, $[n] \setminus V$ gets refined into colour classes each of which consists of vertices with the same number of neighbours in $V$. For $i \in \{0\} \cup \mathbb{N}$, let $W_i \subseteq [n] \setminus V$ be the set of vertices in $[n] \setminus V$ with precisely $i$ neighbours in $V$.

If $|V| \le n/2d$, then we can simply take $U_1 = \bigcup_{i \ge 1} W_i$ (vertices in $[n] \setminus V$ with at least one neighbour in $V$) and $U_2 = [n] \setminus U_1$. Trivially, since the graph is $d$-regular, $|U_1| \le d |V| \le d \cdot (n/2d) = n/2$. On the other hand, since $d \ge n^{1/2+\varepsilon}$, $|U_1| \ge d - |V| \ge d - n / 2d = (1+o(1)) d \ge Cn/d$. Hence, the partition $[n]=U_1 \sqcup U_2$ satisfies the desired property.

To deal with sets $V$ satisfying $n / 2d < |V| < Cn/d$, we will take $U_1 = \bigcup_{i \ge t} W_i$ for some threshold value $t$ that will depend on $|V|$. The value of $t$ will be carefully selected so that $U_1$ has size at most $n/2$ but at least $\varepsilon n$ for some $\varepsilon > 0$. To bound the size of $U_1$, we will use the ``sandwich theorem'' (\cref{thm:sandwich}). Let $p_- = 0.9 d/n$ and $p_+ = 1.1 d/n$. The ``sandwich theorem'' implies that there exists a coupling between two copies of the binomial random graph and a random $d$-regular graph such that whp $\mathcal{G}(n,p_-) \subseteq \mathcal{G}_{n,d} \subseteq \mathcal{G}(n,p_+)$. Since we aim for a statement that holds whp, we may use this coupling to get the desired bounds. 

For a given set $V \subseteq [n]$ of size $v$ with $n/2d < v < Cn/d$, let $X_+$ be the number of vertices in $[n] \setminus V$ in $\mathcal{G}(n,p_+)$ that are adjacent to at least $t := 4 dv/n$ vertices in $V$, noting that $X_+ \geq |U_1|$. For a given vertex $y \in [n] \setminus V$, the number of neighbours of $y$ in $V$ is $Y \sim \textrm{Bin}(v, p_+)$ with the expectation equal to $\lambda = 1.1dv/n = \Theta(1)$; note that $0.55 < \lambda < 1.1C$. Using \eqref{eq:chern_lb} we get that
$$
q_+ := \mathbb{P} \left( Y \ge t \right) = \mathbb{P} \left( Y \ge \frac {4 \lambda}{1.1} \right) \le \exp \left( - \frac {(4/1.1-1)^2}{2(1+(4/1.1-1)/3)} \lambda \right) < e^{-1.84 \lambda} < e^{-1}.
$$
On the other hand, since $\lambda = \Theta(1)$, $Y$ tends to a Poisson random variable with parameter $\lambda$ in distribution. As a result, $q_+  > \xi_+$ for some universal constant $\xi_+ > 0$ and so $q_+ = \Theta(1)$. 

Coming back to $X_+$, let us note that $X_+ \sim \textrm{Bin}(n-v, q_+)$ with $\mathbb{E} X_+ = (n-v)q_+ = (1+o(1)) nq_+ = \Theta(n)$ (here is the place where we use the fact that $q_+ > \xi_+ > 0$). It follows immediately from Chernoff's bound~\eqref{eq:chern_lb} that
\begin{eqnarray*}
\mathbb{P} \left( X_+ \ge \frac {n}{2} \right) &\le& \exp \left( - \Theta(n) \right). 
\end{eqnarray*}
By the union bound over 
$$
\sum_{n/2d < v < Cn/d} \binom{n}{v} = O \left( \frac {n}{d} \right) \ \left( \frac {en}{Cn/d} \right)^{Cn/d} = O \left( \frac {n}{d} \right) \ \exp \left( \Theta( n \log d / d ) \right) = \exp \left( O ( n^{1/2-\varepsilon} \log n  ) \right)
$$ 
choices for $V$, we get the desired universal upper bound of $n/2$ for all $X_+=X_+(V)$ in $\mathcal{G}(n,p_+)$ that holds whp. Since $\mathcal{G}_{n,d} \subseteq \mathcal{G}(n,p_+)$ whp, this property also holds whp in $\mathcal{G}_{n,d}$ for all $U_1=U_1(V)$.

A symmetric argument can be used to show a lower bound for $|U_1|$. For a given set $V \subseteq [n]$ of size $v$, $n/2d < v < Cn/d$, let $X_- \leq |U_1|$ be the number of vertices in $[n] \setminus V$ in $\mathcal{G}(n,p_-)$ that are adjacent to at least $t = 4 dv/n$ vertices in $V$. Arguing as before, we get that $X_- \sim \textrm{Bin}(n-v, q_-)$ for some $q_- > \xi_- >0$. As a result, $\mathbb{E} X_- = (n-v)q_- = (1+o(1)) nq_- = \Theta(n)$ and we get concentration with exponentially small failure probability. The union bound over all choices of $V$ gives us the desired universal lower bound of, say, $n \xi_- / 2 \ge C n/d$ for $X_-$ in $\mathcal{G}(n,p_-)$ for all $V$ that holds whp. This finishes the proof of the claim, since $\mathcal{G}(n,p_-) \subseteq \mathcal{G}_{n,d}$ whp so this lower bound also holds whp in $\mathcal{G}_{n,d}$.
\end{proof}

The proof of \cref{cl:intermediate_large} is more complex. Let us start with the following simple observation. 

\begin{lemma}\label{lem:bound_on_k}
Let $\varepsilon\in(0,1/2)$, $n^{1/2+\varepsilon} \leq d=d(n)\leq n/2$, and let $\mathbf{G}_n \sim \mathcal{G}_{n,d}$. 
For any fixed $\delta > 0$, the following property holds whp: for any $U \subseteq [n]$ of size $u = |U| > \delta n$ and any $V \subseteq [n] \setminus U$ of size $v = |V| \ge 500 n / d \delta$, the number of edges $e(U,V)$ between $U$ and $V$ satisfies the following bounds
$$
0.8 u v \cdot \frac {d}{n-1} \le e(U,V) \le 1.2 u v \cdot \frac {d}{n-1} \,.
$$
\end{lemma}

\begin{proof}
Let $p_- = 0.9 d / (n-1)$ and $p_+ = 1.1 d / (n-1)$. The ``sandwich theorem'' (\cref{thm:sandwich}) implies that there exists a coupling between two copies of the binomial random graph and a random $d$-regular graph such that whp $\mathcal{G}(n,p_-) \subseteq \mathcal{G}_{n,d} \subseteq \mathcal{G}(n,p_+)$. Since we aim for a statement that holds whp, we may use this coupling to get the desired bounds.

Fix any $U \subseteq [n]$ of size $u = |U| > \delta n$ and any $V \subseteq [n] \setminus U$ of size $v = |V| \ge 500 n/d \delta$. In particular, note that $uvd > 500 n^2$. Let $X_-$ be the number of edges between $U$ and $V$ in $\mathcal{G}(n,p_-)$. Clearly, $X_- \sim \textrm{Bin}(uv,p_-)$ with $\mathbb{E} X_- = 0.9 uv d/(n-1)$. It follows immediately from Chernoff's bound~\eqref{eq:chern_ub} that
\begin{eqnarray*}
\mathbb{P} \left( X_- \le 0.8 uv \cdot \frac {d}{n-1} \right) &\le& \exp \left( - \frac {0.1^2}{2 \cdot 0.9} \, uv \cdot \frac {d}{n-1} \right) ~\le~ \exp \left( - \frac {uvd}{250 n} \right) ~\le~ e^{-2n}.
\end{eqnarray*}
By the union bound over at most $(2^{n})^2$ choices for $U$ and $V$, we get the desired lower bound for $e(U,V)$ (for any pair $U$, $V$) in $\mathcal{G}(n,p_-)$ that holds whp. Since $\mathcal{G}(n,p_-) \subseteq \mathcal{G}_{n,d}$ whp, this property also holds whp in $\mathcal{G}_{n,d}$.

A symmetric argument can be used to get the desired upper bound. Indeed, if $X_+$ is the number of edges between $U$ and $V$ in $\mathcal{G}(n,p_+)$, then $\mathbb{E} X_+ = 1.1 uv d/(n-1)$ and Chernoff's bound~\eqref{eq:chern_lb} gives us that
\begin{eqnarray*}
\mathbb{P} \left( X_+ \ge 1.2 uv \cdot \frac {d}{n-1} \right) &\le& \exp \left( - \frac {0.1^2}{2 (1.1 + 0.1/3)} \, uv \cdot \frac {d}{n-1} \right) ~\le~ \exp \left( - \frac {uvd}{250 n} \right).
\end{eqnarray*}
The upper bound holds whp for $\mathcal{G}(n,p_+)$ and so it holds also whp for $\mathcal{G}_{n,d}$ since whp $\mathcal{G}_{n,d} \subseteq \mathcal{G}(n,p_+)$.
\end{proof}

\begin{remark}
\cref{lem:bound_on_k} also follows from the Expander Mixing Lemma and~\cref{th:lambda}. 
\end{remark}

Now, we can move to the proof of \cref{cl:intermediate_large}.

\begin{proof}[Proof of \cref{cl:intermediate_large}]
Fix any $\delta \in (0,1]$ and let $C = C(\delta)$ be a large enough constant that will be specified later. In particular, we will assume that $C \ge 500 / \delta$ so that we may apply \cref{lem:bound_on_k}. 

Suppose that there exists a partition $[n]=V \sqcup ([n] \setminus V)$ with $Cn/d \le |V| \le n/2$ such that after one round of CR there exists a colour class $U$ of size more than $\delta n$. Note that this implies that every vertex in $U$ has the same number of neighbours in $V$ (hence every vertex in $U$ also has the same number of neighbours in $[n] \setminus V$). If $U \subseteq V$, then it will be convenient to concentrate on the number of neighbours in $[n] \setminus V \subseteq [n] \setminus U$ but if $U \subseteq [n] \setminus V$, then we will concentrate on the number of neighbours in $V \subseteq [n] \setminus U$. Our goal is to estimate the probability of the weaker but necessary property that there exists a pair of sets $(U,V)$ such that $V \subseteq [n] \setminus U$, $|V| \ge C n / d$, $|U| > \delta n$, and every vertex in $U$ has the same number of neighbours in $V$.

Fix $V \subseteq [n]$ and $U \subseteq [n] \setminus V$ such that $v:=|V| \ge C n/d$ and $u:= |U| > \delta n$. Note that, in particular, $v < (1-\delta)n$. For each non-negative integer $k\leq v$, define $\mathcal{E}_k(U,V)$ to be the event that every vertex in $U$ has exactly $k$ neighbours in $V$, in which case the number of edges between $U$ and $V$ is exactly $uk$. By \cref{lem:bound_on_k}, since we aim for a statement that holds whp, we may assume that the number of edges between $U$ and $V$ is at least $0.8 uv d / (n-1)$ and at most $1.2 uv d / (n-1)$. Hence, we may restrict to considering $k$ such that 
\begin{equation}\label{eq:bounds_on_k}
0.8 v \cdot \frac {d}{n-1} ~\le~ k ~\le~ 1.2 v \cdot \frac {d}{n-1}.
\end{equation}

Now, note that the expected number of edges induced by $[n] \setminus V$ is $\binom{n-v}{2} \cdot \frac {d}{n-1}$. We will show that it is highly unlikely that the actual number deviates substantially from it. Let 
$$
m_- = 0.9 \binom{n-v}{2} \cdot \frac {d}{n-1} \qquad \text{ and } \qquad m_+ = 1.1 \binom{n-v}{2} \cdot \frac {d}{n-1}. 
$$
By \cref{thm:newcount} and the Stirling's formula ($s! = (1+o(1)) \sqrt{2 \pi s} (s/e)^s$), letting $\mathbf{d}:=(d,\ldots,d)\in\mathbb{Z}^n$, the number of $d$-regular graphs on $[n]$ can be estimated as follows:
\begin{eqnarray}
g(\mathbf{d}) = \Theta \left( \frac{{n-1\choose d}^n}{ \sqrt{dn} {{\binom{n}{2}}\choose dn/2}} \right) &=& \left( (1+o(1)) \sqrt{ \frac {n}{2\pi d (n-d)} } \cdot \frac{ \left( \frac {n-1}{e} \right)^{n-1} } { \left( \frac {d}{e} \right)^{d} \left( \frac {n-1-d}{e} \right)^{n-1-d} } \right)^n \label{eq:Gnd} \\
&& \qquad \cdot \left( \frac{ \Theta(1) \left( \frac {n(n-1)/2}{e} \right)^{n(n-1)/2} } { \left( \frac {dn/2}{e} \right)^{dn/2} \left( \frac {n(n-1-d)/2}{e} \right)^{n(n-1-d)/2} } \right)^{-1} \nonumber \\
&=& \Theta(1) \cdot {{\binom{n}{2}}\choose dn/2} \left( (1+o(1)) \sqrt{ \frac {n}{2\pi d (n-d)} } \right)^n = {{\binom{n}{2}}\choose dn/2} d^{-\Theta(n)}. \nonumber 
\end{eqnarray}
Hence, the probability that the number of edges induced by $[n]\setminus V$ is at most $m_-$ or at least $m_+$ can be upper bounded by
\begin{align*}
\sum_{\substack{m\leq m_- \\ m \ge m_+}} \frac {{{n-v\choose 2}\choose m}{{\binom{n}{2}}-{n-v\choose 2}\choose dn/2-m}}{g(\mathbf{d})}
& = n^{\Theta(n)} \sum_{\substack{m\leq m_- \\ m \ge m_+}} \frac{{{n-v\choose 2}\choose m}{{\binom{n}{2}}-{n-v\choose 2}\choose dn/2-m}}{{{\binom{n}{2}}\choose dn/2}}\\
&=
n^{\Theta(n)}\cdot\mathbb{P} \left( \eta\leq m_- \text{ or } \eta \ge m_+ \right),
\end{align*}
where $\eta$ is the hypergeometric random variable with parameters $\binom{n}{2}$, $\binom{n-v}{2}$, and $dn/2$. Clearly, 
$$
 \mathbb{E}\eta = \frac {dn}{2} \cdot \frac { \binom{n-v}{2} }{ \binom{n}{2} } = {n-v\choose 2}\frac{d}{n-1} = \Theta(dn).
$$
By Chernoff's bound for the hypergeometric distribution (see, \eqref{eq:chern_lb}, \eqref{eq:chern_ub}, and~\cref{rk:hypergeometric}),
$$
 \mathbb{P}(\eta\leq m_- \text{ or } \eta \ge m_+) = \mathbb{P} \left( |\eta - \mathbb{E}\eta| \ge 0.1 \mathbb{E}\eta \right) = \exp\left(-\Omega\left( \mathbb{E}\eta \right)\right)=\exp(-\Omega(dn)).
$$

Similarly, if $|V| \ge n^{3/4}$, then the expected number of edges induced by $V$ is $\binom{v}{2} \cdot \frac {d}{n-1} = \Theta( v^2 d / n)$ and we get that with probability $\exp(-\Omega(dv^2/n))=\exp(-\Omega(d n^{1/2}))$, the number of edges induced by $V$ is at most $0.9 \binom{v}{2} \cdot \frac {d}{n-1}$ or at least $1.1 \binom{v}{2} \cdot \frac {d}{n-1}$.

It remains to concentrate on the case when the number of edges induced by $[n]\setminus V$ is at least $m_-$ but at most $m_+$, that is, when the average degree of the graph induced by $[n]\setminus V$ is at least $0.9 (n-1-v) \cdot \frac {d}{n-1}$ but at most $1.1 (n-1-v) \cdot \frac {d}{n-1}$. Let us first deal with the case when $|V| \ge n^{3/4}$ so we may additionally assume that the average degree of the graph induced by $V$ is at least $0.9 (v-1) \cdot \frac {d}{n-1}$ but at most $1.1 (v-1) \cdot \frac {d}{n-1}$.

By \cref{thm:newcount} and \cref{lm:even},
\begin{equation}
 \mathbb{P}(\mathcal{E}_k(U,V)) = e^{-\Omega(dn)} +e^{-\Omega(d n^{1/2})} + O\left(\sum_{D} h(k,u,v,D) \right),
\label{eq:U-V-count}
\end{equation}
where $D$ denotes the number of edges between $V$ and $[n] \setminus (V \cup U)$, and
$$
h(k,u,v,D) := \frac{{v\choose k}^u{(n-v-u)v\choose D} \binom{v-1}{d-(D+ku)/v}^v {n-1-v\choose d-(D+ku)/(n-v)}^{n-v}}{ \binom{{v \choose 2}}{(dv-ku-D)/2} {{n-v\choose 2}\choose (d(n-v)-ku-D)/2} g(\mathbf{d}) } \,.
$$
Indeed, there are at most ${v\choose k}^u$ ways to place edges between $U$ and $V$, and at most ${(n-v-u)v\choose D}$ ways to place edges between $V$ and $[n] \setminus (V \cup U)$. (Note that these values are trivial upper bounds but not the exact ones as some choices create vertices of degree more than $d$.) It remains to estimate the number of graphs induced by the set $[n] \setminus V$ and the number of graphs induced by the set $V$. Importantly, once other edges are fixed, these graphs have a fixed degree distribution. In particular, the average degree of the graphs induced by $[n] \setminus V$ is precisely $f(D) := d-(D+ku)/(n-v)$. Similarly, the average degree of the graphs induced by $V$ is $\tilde f(D) := d-(D+ku)/v$. Hence, we may use \cref{thm:newcount} and \cref{lm:even} to get upper bounds for the number of such graphs, which in turn imply~\eqref{eq:U-V-count}. (Let us point out that $f(D)$ and $\tilde f(D)$ are not necessarily integers. However, to keep the notation simple, we write ${n-1-v\choose f(D)}^{n-v}$ instead of the product of $n-v$ terms, each of them being ${n-1-v\choose \lfloor f(D) \rfloor}$ or ${n-1-v\choose \lceil f(D) \rceil}$.) Finally, since the average degree of the graph induced by $[n]\setminus V$ and the one induced by $V$ are restricted, $D$ satisfies the requirements 
\begin{eqnarray}
0.9 (n-1-v) \cdot \frac {d}{n-1} \le &f(D)& \le 1.1 (n-1-v) \cdot \frac {d}{n-1} \label{eq:bounds_on_D1}, \\
0.9 (v-1) \cdot \frac {d}{n-1} \le & \tilde f(D)& \le 1.1 (v-1) \cdot \frac {d}{n-1}. \label{eq:bounds_on_D2}
\end{eqnarray}

There are three binomials in the numerator of $h(k,u,v,D)$ that are raised to powers that are functions of $n$. We need to take advantage of them using \cref{cor:binomial_product} (see~\eqref{eq:hyp2}). By \eqref{eq:bounds_on_D1} and recalling $v<(1-\delta)n$, 
\begin{eqnarray}
{n-1-v\choose f(D)}^{n-v} &\le& \left( \frac {(1+o(1))(n-1-v)}{f(D) (n-1-v-f(D))} \right)^{(n-v-1)/2} \binom{(n-v)(n-v-1)}{d(n-v)-ku-D} \nonumber\\
&\le& \left( \frac {1+o(1)}{0.9 \cdot \delta \cdot d \cdot (1-1.1 \cdot 0.5)} \right)^{(n-v-1)/2} \binom{(n-v)(n-v-1)}{d(n-v)-ku-D} \nonumber\\
&\le& \left( \frac {3}{\delta d} \right)^{(n-v-1)/2} \binom{(n-v)(n-v-1)}{d(n-v)-ku-D}.\label{eq:extra_D1}
\end{eqnarray}
Similarly, by \eqref{eq:bounds_on_D2}, 
\begin{eqnarray}
\binom{v-1}{\tilde f(D)}^v &\le& \left( \frac {(1+o(1))(v-1)}{\tilde f(D) (v-1-\tilde f(D))} \right)^{(v-1)/2} \binom{v(v-1)}{dv-ku-D} \nonumber \\
&\le& \left( \frac {1+o(1)}{0.9 \cdot (vd/n) \cdot (1-1.1 \cdot 0.5)} \right)^{(v-1)/2} \binom{v(v-1)}{dv-ku-D} \nonumber \\
&\le& \left( \frac {3n}{vd} \right)^{(v-1)/2} \binom{v(v-1)}{dv-ku-D}. \label{eq:extra_D2}
\end{eqnarray}
Finally, by \eqref{eq:bounds_on_k}, 
\begin{eqnarray}
\binom{v}{k}^u &\le& \left( \frac {v}{k (v-k)} \right)^{(u-1)/2} \binom{vu}{ku} \nonumber\\
&\le& \left( \frac {1+o(1)}{ 0.8 vd/n (1-1.2 \cdot 0.5)} \right)^{(u-1)/2} \binom{vu}{ku} \nonumber\\
&\le& \left( \frac {4n}{dv} \right)^{(u-1)/2} \binom{vu}{ku}.\label{eq:extra_D3}
\end{eqnarray}

For future reference, let us highlight that \eqref{eq:extra_D2} only holds when $v \ge n^{3/4}$, whilst the other two bounds \eqref{eq:extra_D1} and \eqref{eq:extra_D3} hold in general. Substituting in these three bounds, we get
$$
h(k,u,v,D) = O\left( n^{3/2} \ \frac{ \left( \frac {4n}{dv} \right)^{u/2} \binom{vu}{ku} {(n-v-u)v\choose D} \left( \frac {3n}{vd} \right)^{v/2} \binom{v(v-1)}{dv-ku-D} \left( \frac {3}{\delta d} \right)^{(n-v)/2} \binom{(n-v)(n-v-1)}{d(n-v)-ku-D} }{ \binom{{v \choose 2}}{(dv-ku-D)/2} {{n-v\choose 2}\choose (d(n-v)-ku-D)/2} g(\mathbf{d}) } \right).
$$
Now, by \cref{cor:binomial_product} (see~\eqref{eq:hyp3}),
\begin{multline*}
h(k,u,v,D) \\= O\left( n^{7/2} \ \frac{ \left( \frac {4n}{dv} \right)^{u/2} \binom{vu}{ku} {(n-v-u)v\choose D} \left( \frac {3n}{vd} \right)^{v/2} \binom{{v \choose 2}}{(dv-ku-D)/2} \left( \frac {3}{\delta d} \right)^{(n-v)/2} {{n-v\choose 2}\choose (d(n-v)-ku-D)/2} }{ g(\mathbf{d}) } \right).
\end{multline*}
By \cref{cor:binomial_product} (see~\eqref{eq:hyp0}), we can collect all binomial coefficients together to get
\begin{eqnarray*}
h(k,u,v,D) &=& O\left( n^{7/2} \ \frac{ \left( \frac {4n}{dv} \right)^{u/2} \binom {vu+(n-v-u)v+{v \choose 2}+{n-v\choose 2}} {ku+D+(dv-ku-D)/2+(d(n-v)-ku-D)/2} \left( \frac {3n}{vd} \right)^{v/2} \left( \frac {3}{\delta d} \right)^{(n-v)/2} }{ g(\mathbf{d}) } \right) \\
&=& O\left( n^{7/2} \ \frac{ \left( \frac {4n}{dv} \right)^{u/2} \binom {{n \choose 2}} {dn/2} \left( \frac {3n}{vd} \right)^{v/2} \left( \frac {3}{\delta d} \right)^{(n-v)/2} }{ g(\mathbf{d}) } \right).
\end{eqnarray*}
Using \eqref{eq:Gnd} we get that 
$$
g(\mathbf{d}) \ge \binom {{n \choose 2}} {dn/2} ( (2+o(1)) \pi d)^{-n/2} \ge \binom {{n \choose 2}} {dn/2} (7d)^{-n/2},
$$
and so 
\begin{eqnarray*}
h(k,u,v,D) &=& O\left( n^{7/2} \ \left( \frac {4n}{dv} \right)^{u/2} \left( \frac {3n}{vd} \right)^{v/2} \left( \frac {3}{\delta d} \right)^{(n-v)/2} ( 7 d)^{n/2} \right) \\
&=& O\left( n^{7/2} \ \left( \frac {4n}{dv} \right)^{u/2} \left( \frac {\delta n}{v} \right)^{v/2} \left( \frac { 21 }{\delta} \right)^{n/2} \right) \,.
\end{eqnarray*}
Now, let $h(v) := (\delta n / v)^{v/2}$ and note that $\frac{\partial}{\partial v}\log h = \frac {1}{2} \log\frac{\delta n}{ev}$. Hence, $h(v)$ is maximized for $v = \delta n / e$ and we get that for any positive integer $v$,
$$
\left( \frac {\delta n}{v} \right)^{v/2} \le \max_{v} h(v) = \exp \left( \frac {\delta n}{2e} \right) \le 2^{n/2} \,,
$$
since $\exp(1/e) \approx 1.445 \le 2$ and $\delta\leq 1$. Since $u>\delta n$ and $v\geq Cn/d$, it follows that
\begin{eqnarray*}
h(k,u,v,D) &=& O\left( n^{7/2} \ \left( \frac {4n}{dv} \right)^{u/2} \left( \frac { 42 }{\delta} \right)^{n/2} \right) \\
&=& O\left( n^{7/2} \ \left( \left( \frac {4}{C} \right)^{\delta} \cdot \frac { 42 }{\delta} \right)^{n/2} \right) = O \left( 5^{-n} \right),
\end{eqnarray*}
provided that $C$ is large enough so that, say, 
$\left( \frac {4}{C} \right)^{\delta} \cdot \frac {42}{\delta} \le 42 / 1200 = 0.035 < 0.04 = 1/5^2$.
Since $\delta$ is fixed, this condition can be easily satisfied and we may now finally define the constant~$C$:
$$
C = C(\delta) := \max \left\{ \frac {4}{(\delta / 1200)^{1/\delta}}, \frac {500}{ \delta } \right\}.
$$

We conclude that if $|V| \ge n^{3/4}$, then
$$
 \mathbb{P}(\mathcal{E}_k(U,V)) = e^{-\Omega(dn)} + e^{-\Omega(d n^{1/2})} + O\left(\sum_{D} h(k,u,v,D) \right) = O \left( n^2 \, 5^{-n} \right).
$$
If $v < n^{3/4}$ then, as mentioned earlier, we do not get the term $\left( \frac {3n}{vd} \right)^{(v-1)/2}$ in the estimation of $h(k,u,v,D)$ (see \eqref{eq:extra_D2}). However, for $v < n^{3/4}$, this term does not help us much anyway: $\left( \frac {3n}{vd} \right)^{(v-1)/2} = \exp( - \Omega( v \log v ) ) \ge \exp( - n^{4/5} )$. Hence, regardless of the size of $V$,
$$
 \mathbb{P}(\mathcal{E}_k(U,V)) = O \left( n^2 \exp(n^{4/5}) \, 5^{-n} \right).
$$
Finally, by the union bound,
$$
 \mathbb{P} \left( \exists k,U,V \,\,\mathcal{E}_k(U,V) \right) \leq n \cdot 2^{n} \cdot 2^n \cdot O \left( n^2 \exp(n^{4/5}) \, 5^{-n} \right) = o(1),
$$
which finishes the proof of the theorem. 
\end{proof}

\section{Proof of \cref{thm:main_many_large_set}: Sparse Case}
\label{sc:sparse-many-large}

Sparser graphs clearly require more rounds of CR. Consider any $d$-regular graph with diameter $D$ and let $u$ and $v$ be any two vertices at distance $D$ from each other. CR run on the initial partition $V_1=\{v\}$ and $V_2 = [n] \setminus \{v\}$ requires at least $D-1$ rounds to converge to a discrete colouring. Indeed, after $D-2$ rounds there are at least two vertices at distance at least $D-1$ from $v$ that are still of the same colour. In this section, we prove the following.


\begin{theorem}\label{thm:sparse_many_large_set}
Let $d_0$ be large enough, let $d_0\leq d =o(n)$, and let $\mathbf{G}_n \sim \mathcal{G}_{n,d}$. Let $k\in\mathbb{N}$ be an arbitrary constant. Then, the following holds whp: for every non-trivial partition $[n]=V_1\sqcup V_2$ of the vertex set of $\mathbf{G}_n$, after at most $\diam (\mathbf{G}_n)+2$ rounds of CR, there exists a partition $[n] = V_1 \sqcup \ldots \sqcup V_k$ such that for any $i \in [k]$, $n/3k \le |V_i| \le 3n/k$ and $V_i$ is a union of some colour classes. 
\end{theorem}

We start from auxiliary anti-concentration results that will be used to prove \cref{thm:sparse_many_large_set} in~\cref{sc:proof-thm-2-many-sparse}.

\subsection{Anti-concentration Results}

\begin{lemma}
\label{lm:from_single_set-to_partition-gen}
Let $\ell$ be a large enough constant, and let $d_0 = d_0(\ell)$ be another large enough constant. Let $d_0\leq d=o(n)$ and $\mathbf{G}_n\sim\mathcal{G}_{n,d}$. Then, the following property holds whp: for every set $U$ of size $|U|\in[\frac{n\ell}{d},\frac{n}{2}]$ and every non-negative integer $s$, the number of vertices in $[n]\setminus U$ that have exactly $s$ neighbours in $U$ is at most $10n/\ln \ell$.
\end{lemma}

\begin{proof}



Set $t=10n/\ln\ell$. The Expander Mixing Lemma and \cref{th:lambda} imply that whp, for every set $U$ of size $m\in[n\ell/d,n/2]$ and every set $V\subset[n]\setminus U$ of size $t$,
\begin{itemize}
\item the number of edges between $U$ and $V$ is at most $\min\{dm/4,3td/4\}$,
\item the number of edges induced by $V$ is at most $0.01dt$,
\item the number of edges between $U$ and $[n]\setminus (U\cup V)$ is at least $md/3$,
\item the number of edges between $[n]\setminus(U\cup V)$ and $U\cup V$ is at most $2d(n-m-t)/3$.
\end{itemize}
Due to the first bullet point, it suffices to prove the lemma for $s$ such that $10sn/\ln\ell\leq dm/4$ and $st\leq 3td/4$. Therefore, we fix an integer
\begin{equation}
s\in\left[0,\min\left\{\frac{3d}{4},\frac{dm\ln\ell}{40n}\right\}\right].
\label{eq:s-upper}
\end{equation}
We also fix a set $U$ of size $m\in[\frac{n\ell}{d},\frac{n}{2}]$ and a set $V\subset[n]\setminus U$ of size $t$.
 Let us estimate the probability that every vertex from $V$ has exactly $s$ neighbours in $U$.



Let us order the vertices in $V$ arbitrarily: $x_1,x_2,\ldots,x_t$. Expose the edges in $V$. Let $d-d_i$ be the degree of $x_i$ in $\mathbf{G}_n[V]$. Let $V'\subset V$ be the set of all vertices with $d_i\geq 0.8d$. Due to the second bullet point above, we may assume that $\sum d_i\leq 0.02dt$ and therefore, $|V'|\geq 0.9t$. Denote $t'=|V'|$. Without loss of generality we assume that $V'=\{x_1,\ldots,x_{t'}\}$. 

Let $\mathcal{E}$ be the event that every $x_i$ from $V'$ has $s$ neighbours in $U$. 
 Let $h_{in}\geq d(n-m-t)/6$ and $h_{out}\geq md/3$ be integers such that 
$$
\mathbb{P}(\mathcal{E}\wedge \{|E(\mathbf{G}_n[[n]\setminus (U\cup V)])|=h_{in},\,|E(\mathbf{G}_n[U\times([n]\setminus (U\cup V))]|=h_{out}\})\text{ is maximum}.
$$ 
Let $\Sigma_0$ be the set of all $d$-regular graphs $G$ on $[n]$ satisfying $\mathcal{E}$ and such that $G[V]=\mathbf{G}_n[V]$ and $G[[n]\setminus (U\cup V)]$ and $G[U\times ([n]\setminus(U\cup V))]$ have exactly $h_{in}$ and $h_{out}$ edges, respectively. The following claim completes the proof of \cref{lm:from_single_set-to_partition-gen}.

\begin{claim}
$\mathbb{P}(\mathbf{G}_n\in\Sigma_0\mid\mathbf{G}_n[V])\leq \ell^{-t'/5}$.
\label{cl:large-switchings}
\end{claim}

Indeed, by the union bound over $U,V$ and the number of edges in the graphs $G[[n]\setminus (U\cup V)]$ and $G[U\times ([n]\setminus(U\cup V))]$, we get that probability that there exist sets $U,V$ such that $\mathcal{E}$ holds is at most
\begin{align*}
 o(1)+(nd)^2\sum_{m=n\ell/d}^{n/2}{n\choose m}{n\choose t}\ell^{-t'/5} & \stackrel{t'\leq 0.9t}\leq o(1)+n^4 2^n\left(\frac{en}{t\ell^{9/50}}\right)^t\\
 &\leq o(1)+n^4 2^n \left(\frac{\ln\ell}{ \ell^{9/50}}\right)^{10n/\ln \ell}\leq o(1)+n^4 2^n e^{-n}=o(1).
\end{align*}

\begin{proof}[Proof of \cref{cl:large-switchings}.]
Assume the opposite:  
\begin{equation}
\label{eq:contradiction-ELO}
    \mathbb{P}(\mathbf{G}_n\in\Sigma_0)>\ell^{-t'/5}.
\end{equation}

Let $\mathcal{X}$ be the set of $d$-regular graphs $G$ on $[n]$ such that there exists a vertex $v$ with at least $dm/\ell$ edges between $N(v)$ and $U$. Then, due to \cref{cl:regular-probability-bound},
\begin{equation}
 \mathbb{P}(\mathbf{G}_n\in\mathcal{X})\leq n{dm\choose md/\ell}\left(\frac{2d}{n}\right)^{md/\ell}\leq n\left(\frac{2e\ell d}{n}\right)^{md/\ell}=o(e^{-n}).
 \label{eq:X-small}
\end{equation}

 Fix an integer $\delta$ such that $|\delta|\leq \ell^{0.3}$. Let $\Sigma_1(\delta)$ be the set of $d$-regular graphs $G$ on $[n]$ such that
\begin{itemize}
\item $G[V]=\mathbf{G}_n[V]$,
\item each vertex $x_2,\ldots,x_{t'}$ has exactly $s$ neighbours in $U$,
\item $x_1$ has $s+\delta$ neighbours in $U$, and 
\item there are $h_{in}+\delta$ edges in $G[[n]\setminus (V\cup U)]$ and $h_{out}-\delta$ edges in $G[U\times([n]\setminus (U\cup V))].$
\end{itemize}
We shall prove that 
\begin{equation}
\left|\bigcup_{|\delta|\leq \ell^{0.3}}\Sigma_1(\delta)\right|\geq\ell^{0.2}|\Sigma_0|.
\label{eq:first-step-multiply}
\end{equation}

Let us first compare sizes of $\Sigma_1(1)$ and $\Sigma_0$. Take $G\in\Sigma_0$ and consider a tuple of vertices $(y,u,v)$ such that 
\begin{itemize}
\item $y\notin U\cup V$, $u\in U$, $v\notin U\cup V$,
\item and $\{x_1,y\},\{u,v\}\in E(G)$, $\{x_1,u\},\{y,v\}\notin E(G)$.
\end{itemize}
If we switch
\begin{equation}
 \{x_1,y\},\{u,v\}\mapsto\{x_1,u\},\{y,v\},
\label{eq:switch-def}
\end{equation}
we get a graph from $\Sigma_1(1)$. For every $G\in\Sigma_0\setminus\mathcal{X}$ the number of {\it forward switchings} (i.e., tuples $(y,u,v)$) equals $(d_1-s)(h_{out}-\mu)$, where $\mu\in[0,2dm/\ell]$ since, first, $G\notin\mathcal{X}$ and so the number of edges between $N(y)$ and $U$ is at most $dm/\ell$, and second, the number of edges that touch $N(x_1)\cap U$ is at most $sd=O(d^2 m/n)<dm/\ell$ due to~\eqref{eq:s-upper}. On the other hand, for every graph $G\in\Sigma_1(1)$, the number of {\it backward switchings} equals $(s+1)(h_{in}+1-\nu)$, where $\nu\in[0,2d^2]$. We get
$$
 |\Sigma_0|(d_1-s)(h_{out}-\mu)-|\mathcal{X}\cap\Sigma_0|O(d^2n)=|\Sigma_1(1)|(s+1)(h_{in}+1-\nu).
$$
In a similar way, we compare the sizes of $\Sigma_1(-1)$ and $\Sigma_0$: For $G\in\Sigma_0$, we take a tuple of vertices $(y,u,v)$ such that 
\begin{itemize}
\item $y\in U$, $u,v\notin U\cup V$,
\item and $\{x_1,y\},\{u,v\}\in E(G)$, $\{x_1,u\},\{y,v\}\notin E(G)$,
\end{itemize}
and switch as in~\eqref{eq:switch-def}. We get
$$
 |\Sigma_0|s(h_{in}-\nu')=|\Sigma_1(-1)|(d_1-s+1)(h_{out}+1-\mu')-|\mathcal{X}\cap\Sigma_1(-1)|O(d^2n),
$$
where $\nu'\in[0,2d^2]$, and $\mu'\in[0,2dm/\ell].$ 

If $(d_1-s)h_{out}>s h_{in}$, then 
$$
\frac{|\Sigma_1(1)|}{|\Sigma_0|}>1-O(1/\ell)-O(1/s)-O\left(d^2n\cdot  \frac{|\mathcal{X}|}{|\Sigma_0|}\right)=1-O(1/\ell)-O(1/s)
$$ 
due to~\eqref{eq:X-small} and the assumption~\eqref{eq:contradiction-ELO} and the lower bounds on $h_{out}$ and $h_{in}$. Otherwise, 
\begin{equation}
\frac{|\Sigma_1(-1)|}{|\Sigma_0|}>1-O(1/\ell).
\label{eq:-1}
\end{equation}

First we assume $(d_1-s)h_{out}>sh_{in}$. If $s<\ell^{0.7}$, then $(d_1-s)h_{out}=\Omega(\ell^{0.3}s h_{in})$ since $h_{out}=\Theta(md)$ and $h_{in}=\Theta(dn)$. We conclude that $|\Sigma_1(1)|/|\Sigma_0|=\Omega(\ell^{0.3})>\ell^{0.2}$, completing the proof of~\eqref{eq:first-step-multiply}.

Let $s\geq \ell^{0.7}$. For $\delta=1,\ldots,\ell^{0.3}-1$, we define a switching operation that maps $\Sigma_1(\delta)$ to $\Sigma_1(\delta+1)$ in exactly the same way. We get
$$
 |\Sigma_1(\delta)|(d_1-s-\delta)(h_{out}-\delta-\mu_{\delta})-|\mathcal{X}\cap\Sigma_1(\delta)|O(d^2n)=|\Sigma_1(\delta+1)|(s+\delta+1)(h_{in}+1+\delta-\nu_{\delta}),
$$
for some $\nu_{\delta}\in[0,2d^2]$ and $\mu_{\delta}\in[0,2dm/\ell]$. We then get 
$$
\frac{|\Sigma_1(\delta+1)|}{|\Sigma_1(\delta)|}>1-O(\delta/s)-O(1/\ell)=1-O(\delta/\ell^{0.7}),
$$ 
implying 
$$
\frac{|\Sigma_1(\delta+1)|}{|\Sigma_0|}=\frac{|\Sigma_1(1)|}{|\Sigma_0|}\cdot\prod_{j=1}^{\delta}\frac{|\Sigma_1(j+1)|}{|\Sigma_1(j)|}>\prod_{j=1}^{\delta+1}(1-O(j/\ell^{0.7}))=1-O(\delta^2/\ell^{0.7})=1-O(\ell^{-0.1}).
$$ 
Consequently, we can deduce
$$
 \left|\bigcup_{|\delta|\leq \ell^{0.3}}\Sigma_1(\delta)\right|\geq\left|\bigcup_{\delta=1}^{\ell^{0.3}}\Sigma_1(\delta)\right|\geq (1-O(\ell^{-0.1}))\ell^{0.3}|\Sigma_0|>\ell^{0.2}|\Sigma_0|,
$$
as needed.

Now, let $(d_1-s)h_{out}\leq sh_{in}$. In particular, $s=\Omega(\ell)$. Recall~\eqref{eq:-1}. Similarly, for $\delta=1,\ldots,\ell^{0.3}-1$, we define a switching operation that maps $\Sigma_1(-\delta)$ to $\Sigma_1(-\delta-1)$ and get
$$
 |\Sigma_1(-\delta)|(s-\delta)(h_{in}-\delta-\nu'_{\delta})=|\Sigma_1(-\delta-1)|(d_1-s+\delta+1)(h_{out}+1+\delta-\mu'_{\delta})-|\mathcal{X}\cap\Sigma_1(-\delta-1)|O(d^2n),
$$
for some $\nu'_{\delta}\in[0,2d^2]$ and $\mu'_{\delta}\in[0,2dm/\ell]$. We then get 
$$
\frac{|\Sigma_1(-\delta-1)|}{|\Sigma_1(-\delta)|}>1-O(\delta/s)-O(1/\ell)=1-O(\delta/\ell),
$$ 
implying 
$$
\frac{|\Sigma_1(-\delta-1)|}{|\Sigma_0|}>\prod_{j=1}^{\delta+1}(1-O(j/\ell))=1-O(\delta^2/\ell)=1-O(\ell^{-0.4}).
$$ 
Finally, this gives 
$$
 \left|\bigcup_{|\delta|\leq \ell^{0.3}}\Sigma_1(\delta)\right|\geq\left|\bigcup_{\delta=1}^{\ell^{0.3}}\Sigma_1(-\delta)\right|\geq (1-O(\ell^{-0.4}))\ell^{0.3}|\Sigma_0|>\ell^{0.2}|\Sigma_0|,
$$
as needed.

Let $\Sigma_1:=\bigcup_{|\delta|\leq \ell^{0.3}}\Sigma_1(\delta)$. We now define sets $\Sigma_2,\Sigma_3,\ldots,\Sigma_{t'}$ as follows: for $|\delta|\leq\ell^{0.3}$, let $\Sigma_i(\delta)$
be the set of $d$-regular graphs $G$ on $[n]$ such that
\begin{itemize}
\item $G[V]=\mathbf{G}_n[V]$,
\item each vertex $x_{i+1},\ldots,x_{t'}$ has exactly $s$ neighbours in $U$,
\item $x_i$ has $s+\delta$ neighbours in $U$,
\item for every $j\in[i-1]$, $|N_U(x_j)-s|\leq \ell^{0.3}$,
\item there are $h_{in}+\sum_{j=1}^{i-1}(|N_U(x_j)|-s)+\delta$ edges in $G[[n]\setminus (U\cup V)]$ and $h_{out}-\sum_{j=1}^{i-1}(|N_U(x_j)|-s)-\delta$ edges in $G[U\times([n]\setminus (U\cup V))].$
\end{itemize}
Let $\Sigma_{i}:=\bigcup_{|\delta|\leq \ell^{0.3}}\Sigma_{i}(\delta)$. The proof that $
\left|\Sigma_{i+1}\right|\geq \ell^{0.2}|\Sigma_i|$
is identical to that of~\eqref{eq:first-step-multiply} since
$$
 \sum_{j=1}^{i-1}(|N_U(x_j)|-s)\leq t\ell^{0.3}=O(h_{out}/(\ell^{0.7}\ln\ell)),
$$
and so it does not contribute significantly to $\nu_{\delta},\mu_{\delta},\nu'_{\delta},$ and $\mu'_{\delta}$. Therefore,
$$
 |\Sigma_0|\leq \ell^{-0.2}|\Sigma_1|\leq \ell^{-0.4}|\Sigma_2|\leq\ldots\leq \ell^{-0.2t'}|\Sigma_{t'}|
$$
--- a contradiction which completes the proof of the claim.
\end{proof}
The proof of the claim was the last piece of the puzzle in establishing the lemma.
\end{proof}

\begin{lemma}
\label{lem:exactly_s_neighbours}
Let $\ell$ be a large enough constant, and let $d_0 = d_0(\ell)$ be another large enough constant. Let $d_0\leq d=o(n)$ and $\mathbf{G}_n\sim\mathcal{G}_{n,d}$. Then, the following property holds whp: for every set $U$ of size $|U|\in[\frac{n}{2d},\frac{n\ell}{d}]$ and every integer $s$ such that $1 \le s \le \ell$, there are at most $0.999n$ vertices that have exactly $s$ neighbours in $U$.
\end{lemma}

\begin{proof}
Let $\varepsilon=0.001$. Let $\mathcal{E}$ be the event that, for any disjoint sets $V$ and $W$ of size $n/4$ and $\varepsilon n$, respectively, the number of edges in $V\cup W$ that touch $V$ is at most $dn/30$. By the Expander Mixing Lemma and \cref{th:lambda} the event $\mathcal{E}$ holds whp. 

Fix a set $U$ of size $m\in[\frac{n}{2d},\frac{\ell n}{d}]$ and a set $V=\{x_1,\ldots,x_t\}\subset[n]\setminus U$ of size $n(1-\varepsilon)$. Divide $V=V'\sqcup V''$, where $V'$ consists of the first $n/4$ vertices. Expose edges inside $V'$ and between $V'$ and $[n]\setminus (U\cup V)$ and assume that $E:=E(\mathbf{G}_n[V'])\cup E(\mathbf{G}_n[V'\times([n]\setminus(U\cup V))])$ has size at most $dn/30$. Let $\tilde V'\subset V'$ be the set of vertices that belong to at most $d/2$ edges in $E$. Clearly, $|\tilde V'|\geq n/12$. Without loss of generality, we assume $\tilde V'=\{x_1,\ldots,x_{t'}\}$, where $t'\geq n/12$. Note that each $x_i$, $i\in[t']$, has at least $d/2-s$ neighbours in $V''$.

Let $\Sigma_0$ be the set of $d$-regular graphs $G$ on $[n]$ such that 
\begin{equation}
\label{eq:exposed-edges}
    E(G[V'])\cup E(G[V'\times ([n]\setminus(U\cup V))])=E
\end{equation} 
and each vertex in $V$ has exactly $s$ neighbours in $U$. 
Let $\Sigma_1$ be the set of $d$-regular graphs $G$ on $[n]$ such that~\eqref{eq:exposed-edges} is satisfied and
\begin{itemize}
\item each vertex $x_2,\ldots,x_t$ has exactly $s$ neighbours in $U$, except for some $x_i\in V''$, whereas $x_i$ has $s-1$ neigbhours in $U$,
\item $x_1$ has $s+1$ neighbours in $U$.
\end{itemize}
We shall prove that $|\Sigma_1|\geq |\Sigma_0|/3$.
Take $G\in\Sigma_0$ and consider a tuple of vertices $(y,u,v)$ such that 
\begin{itemize}
\item $y\in V''$, $u\in U$, $v\in V''$,
\item and $\{x_1,y\},\{u,v\}\in E(G)$, $\{x_1,u\},\{y,v\}\notin E(G)$.
\end{itemize}
If we switch
\begin{equation}
 \{x_1,y\},\{u,v\}\mapsto\{x_1,u\},\{y,v\},
\label{eq:switch-def-2}
\end{equation}
we get a graph from $\Sigma_1$. For every $G\in\Sigma_0$ the number of {\it forward switchings} is at least $(d/2-s)(((3/4-\varepsilon)n-d)s-sd)$: there are at least $d/2-s$ choices of $y$, at least $(3/4-\varepsilon)n-d$ choices of $v\in V''$ that is not adjacent to $y$, and, therefore, at least $((3/4-\varepsilon)n-d)s$ choices of the edge $\{u,v\}$ where $u\in U$. It then remains to subtract at most $sd$ edges $\{u,v\}$ where $u\in U$ is a neighbour of $x_1$. On the other hand, for every graph $G\in\Sigma_1$, the number of {\it backward switchings} is at most $(s+1)(3n/4)d$, with a room for improvement since the choice of $v=x_i$ is actually unique. We get
$$
 |\Sigma_0|(d/2-s)(((3/4-\varepsilon)n-d)s-sd)\leq|\Sigma_1|(s+1)(3n/4)d
$$
implying $|\Sigma_1|\geq|\Sigma_0|/3$, as desired. 

We now let $\Sigma_2$ be the set of $d$-regular graphs $G$ on $[n]$ such that~\eqref{eq:exposed-edges} is satisfied and
\begin{itemize}
\item each vertex $x_3,\ldots,x_t$ has exactly $s$ neighbours in $U$, except for some $x_{i_1},x_{i_2}\in V''$ that have $s-1$ neighbours in $U$,
\item $|N_U(x_1)|\in[s,s+1]$,
\item $x_2$ has $s+1$ neighbours in $U$.
\end{itemize}
Take $G\in\Sigma_0\cup\Sigma_1$ and consider a tuple of vertices $(y,u,v)$ such that 
\begin{itemize}
\item $y\in V''$, $u\in U$, $v\in V''$, and $|N_U(v)|=s$, 
\item $\{x_2,y\},\{u,v\}\in E(G)$, $\{x_2,u\},\{y,v\}\notin E(G)$.
\end{itemize}
If we switch as in~\eqref{eq:switch-def-2} with $x_2$ instead of $x_1$, then
we get a graph from $\Sigma_2$. As above, for every $G\in\Sigma_0\cup\Sigma_1$ the number of {\it forward switchings} is at least $(d/2-s)(((3/4-\varepsilon)n-d-1)s-sd)$. On the other hand, for every graph $G\in\Sigma_1$, the number of {\it backward switchings} is at most $(s+1)(3n/4)d$, as above. We get
$$
 |\Sigma_0\cup\Sigma_1|(d/2-s)(((3/4-\varepsilon)n-d-1)s-sd)\leq|\Sigma_2|(s+1)(3n/4)d
$$
implying $|\Sigma_2|\geq|\Sigma_0\cup\Sigma_1|/3$, as well.

Similarly, we define $\Sigma_3,\ldots,\Sigma_{n/12}$. For the $i$-th set $\Sigma_i$, we get that
$$
 |\Sigma_0\cup\ldots\cup\Sigma_{i-1}|(d/2-s)(((3/4-\varepsilon)n-d-(i-1))s-sd)\leq|\Sigma_i|(s+1)(3n/4)d,
$$
implying $|\Sigma_i|\geq|\Sigma_0\cup\ldots\cup\Sigma_{i-1}|/3$ for all $i\geq 1$. It is easy to see by induction that $|\Sigma_i|\geq(4/3)^{i-1}|\Sigma_0|/3$ for all $i\geq 1$. 
Thus, we get 
$$
 |\Sigma_{n/12}|>(4/3)^{n/12-1}|\Sigma_0|/3,
$$
implying 
$$
\mathbb{P}(\mathbf{G}_n\in\Sigma_0\mid E(\mathbf{G}_n[V'])\cup E(\mathbf{G}_n[V'\times([n]\setminus(U\cup V))])=E)<4\cdot(4/3)^{-n/12}.
$$
Denote $\mathbf{E}_n(U,V):=E(\mathbf{G}_n[V'])\cup E(\mathbf{G}_n[V'\times([n]\setminus(U\cup V))])$. The union bound over $U$ and $V$ gives us that 
\begin{multline*}
 \mathbb{P}(\neg\mathcal{E})+{n\choose\varepsilon n}{n\choose \ell n/d}\sum_{E:\,|E|\leq dn/15}\mathbb{P}(\mathbf{G}_n\in\Sigma_0\mid \mathbf{E}_n(U,V)=E)\cdot\mathbb{P}(\mathbf{E}_n(U,V)=E)\\
 =o(1)+e^{(\varepsilon\ln(e/\varepsilon)+o_d(1))n}(4/3)^{-n/12}=o(1),
\end{multline*}
which completes the proof of the lemma. 
\end{proof}

\subsection{Proof of \cref{thm:sparse_many_large_set}}
\label{sc:proof-thm-2-many-sparse}

With \cref{lm:common_expansion_Thm3,,lm:from_single_set-to_partition-gen,,lem:exactly_s_neighbours} at hand, we can easily prove \cref{thm:sparse_many_large_set}.

\begin{proof}[Proof of \cref{thm:sparse_many_large_set}]
Fix a large enough\footnote{Clearly, if the statement of the theorem is true for large $k$, then it is also true for all $k$.} $k \in \mathbb{N}$, and let $\ell = e^{30k}$, so that the conclusions of~\cref{lm:from_single_set-to_partition-gen,,lem:exactly_s_neighbours} hold. Let $c > 0$ be a small enough constant as in \cref{lm:common_expansion_Thm3}. Let $d_0 = d_0(\ell)$ be a large enough constant as in \cref{lm:common_expansion_Thm3,,lm:from_single_set-to_partition-gen,,lem:exactly_s_neighbours}. Moreover, we will adjust constants $c$ or $d_0$, if needed, for some of the claims below to hold. Since we aim for the statement that holds whp, we may assume that the statements in \cref{lm:common_expansion_Thm3,,lm:from_single_set-to_partition-gen,,lem:exactly_s_neighbours} hold deterministically. 

Consider any non-trivial partition $[n]=V_1\sqcup V_2$. Our goal is to show that after at most $\diam (\mathbf{G}_n)+2$ many rounds of CR, we get a partition into colour classes in which all colour classes have size at most $n / 3k$. To get the desired partition into $k$ parts, each of size at least $n/3k$ but at most $3n/k$, one can iteratively merge colour classes as we did in the proof of \cref{thm:dense_two_rounds}.

Let $u = \min\{ |V_1|, |V_2| \}$ and let $U$ be a colour class of size $u$. Suppose first that $u < cn/d$. Let $r$ be the largest integer such that $ud(d-1)^{r-1} \le cn$. We may adjust $c$ and $d$, if needed, to make sure that $1-100c-4\ln d/d \ge 1/2$, which, in particular, implies that $c \le 1/200$. It follows from \cref{lm:from_single_set-to_partition-gen} that 
$$
|S_r(U)| \ge (1-100c-4\ln d/d)ud(d-1)^{r-1} > \frac {cn}{2d},
$$
and clearly $|S_r(U)| \le ud(d-1)^{r-1} \le cn \le n/2$. After $r \le \diam (\mathbf{G}_n)-1$ rounds of CR, $S_r(U)$ is a union of some colour classes. We may marge them together at this point and continue the process from there.

Suppose now that $U$ is a colour class of size $u = |U| \in [\frac {cn}{2d}, \frac {n}{2d}]$. Let $U'$ be an arbitrary subset of $U$ of size $\frac {cn}{2d}$. On the one hand, trivially, $|S_1(U)| \le d|U| \le n/2$. On the other hand, it follows from \cref{lm:from_single_set-to_partition-gen} that $|S_1(U')| \ge |U'|d/2 = cn/4$ which implies that $|S_1(U)| \ge |S_1(U')| - |U| \ge cn/4 - n/(2d) \ge n\ell/d$, provided that $d$ is large enough. Since $S_1(U)$ is a union of some colour classes, we may merge them into one large colour class and continue from there. 

Suppose this time that $U$ is a colour class of size $u = |U| \in [\frac {n}{2d}, \frac {n\ell}{d}]$. We may adjust $d$, if needed, to make sure $u \le \frac {n\ell}{d} \le \frac {n}{2(\ell+1)}$. After one round of CR, $[n] \setminus U$ is partitioned into sets $W_i$ ($i \in \mathbb{N} \cup \{0\}$); set $W_i$ consists of vertices with exactly $i$ neighbours in $U$. Let $A = \bigcup_{i \le \ell} W_i$ and let $B = \bigcup_{i \ge \ell+1} W_i$. Clearly, $|U|+|A|+|B|=n$. Note that, on the one hand, the number of edges between $U$ and its complement is at least $|B|(\ell+1)$. On the other hand, it is trivially at most $|U|d \le n \ell$. We conclude that 
$$
|B| \le \frac {\ell}{\ell+1} n = \left( 1 - \frac {1}{\ell+1} \right) n,
$$
and so
$$
|A| = n - |B| - |U| \ge \frac {n}{\ell+1} - u \ge \frac {n}{2(\ell+1)}.
$$

Our goal is to show that one can always merge some sets $W_i$ together to get a colour class of size at most $n/2$ but at least $\frac {n}{2(\ell+1)}$, which is at least $\frac {n\ell}{d}$, provided that $d$ is large enough. To that end, we will consider a few cases. If $|A| \le n/2$, then we can simply take the entire set $A$ for the desired colour class. If $|A| > n/2$ but $|A| \le (1-1/(\ell+1))n$, then we may take the entire set $B$ since $|B| = n - |A| - |U| \ge n / (\ell+1) - u \ge n / (2(\ell+1))$ and, trivially, $|B| = n - |A| - |U| < n/2$. 

It remains to concentrate on the case when $|A| \ge (1-1/(\ell+1))n$. Suppose first that $|W_i| \ge n/2$ for some $0 \le i \le \ell$. It follows from \cref{lem:exactly_s_neighbours} that $|W_i| \le 0.999n$. Then, we can take $A \setminus W_i$ for the desired colour class since, trivially, $|A \setminus W_i| \le n - |W_i| \le n/2$ and 
$$
|A \setminus W_i| \ge |A| - |W_i| \ge 0.001 n - \frac {n}{\ell+1} \ge \frac {n\ell}{d},
$$
provided that $d$ is large enough. If $n/4 \le |W_i| < n/2$ for some $0 \le i \le \ell$, then we may simply take $W_i$ as our colour class. Suppose then that $|W_i| < n/4$ for all $0 \le i \le \ell$.  Then, we may start with set $A$ and remove $W_i$'s, one by one, and at some point we get a set of size at most $n/2$ but at least $n/4$. 

Finally, suppose that $U$ is a colour class of size $u = |U| \in [\frac {n\ell}{d}, \frac {n}{2}]$. It follows immediately from \cref{lm:from_single_set-to_partition-gen} that after one round of CR, the complement of $U$ is partitioned into sets of size at most $10n/\ln \ell = n/3k$. We can group some of them together to get a colour class of size at least $n/3k \ge n \ell / d$ but at most $2n/3k \le n/2$ to make sure that after one more round $U$ is also partitioned into sets of size at most $n/3k$. This completes the proof of the theorem.
\end{proof}


\section{Proof of \cref{thm:main_complete_refinement}: Dense Case}
\label{sc:dense-full}

Here, we prove the following.

\begin{theorem}\label{thm:dense-full}
Let $\varepsilon>0$, let $n^{1/2+\varepsilon}\leq d\leq n/2$, and let $\mathbf{G}_n \sim \mathcal{G}_{n,d}$. 
There exists some universal large constant $k\in\mathbb{N}$ such that the following holds whp: for every partition $[n] = V_1 \sqcup \ldots \sqcup V_k$ of the vertex set of $\mathbf{G}_n$ such that for any $i \in [k]$, $n/3k \le |V_i| \le 3n/k$, after three rounds of CR, there are only singleton colour classes.
\end{theorem}

We start from an auxiliary lemma. 

\begin{lemma}\label{lem:neighbourhood}
Let $\varepsilon>0$, $n^{-1/2+\varepsilon}\leq d=d(n)\leq n/2$, and let $\mathbf{G}_n \sim \mathcal{G}_{n,d}$. 
 For every pair of vertices $u,v\in[n]$, let $M_{u,v}\subset N(u)$ and $M'_{u,v}\subset N(v)\setminus N(u)$ be sets of size $\lfloor n/(20d)\rfloor$ chosen uniformly at random. Then the following events hold whp for any pair of vertices $u,v$ in $\mathbf{G}_n$:
\begin{enumerate}
    \item $|N(u)\cap N(v)| < 2d/3$;
    \item $|N(u,v)|< 4n/5$;
    \item $|N(M'_{u,v})\setminus(N(\{u,v\})\cup N(M_{u,v}))|>n/25$ whenever $d\leq n/20$.
\end{enumerate}
\end{lemma}

\begin{proof}
Let $p_- = 0.9 d / (n-1)$ and $p_+ = 1.1 d / (n-1)$. \cref{thm:sandwich} implies that there exists a coupling between two copies of the binomial random graph and a random $d$-regular graph such that whp $\mathcal{G}(n,p_-) \subseteq \mathcal{G}_{n,d} \subseteq \mathcal{G}(n,p_+)$. Since we aim for a statement that holds whp, we may use this coupling to obtain the first two bounds. It is sufficient to show that for a fixed pair of vertices $u,v$ any of the first two events from the assertion of the claim fails in both $G(n,p_-)$ and $G(n,p_+)$ with probability $o(n^2)$.

First we consider the probability that $N(u) \cap N(v)$ is large. In both binomial random graphs the number of shared neighbours is dominated by $\textrm{Bin}(n-1,p_+^2)$. Since $(n-1)p_+^2\leq 0.61d$, the Chernoff bound~\eqref{eq:chern_lb} implies that $|N(u)\cap N(v)|\geq 2d/3$ in $G(n,p_+)$ with probability $\exp(-\Omega(d^2/n))$.

Similarly $N(u)\cup N(v)$ is dominated by $\textrm{Bin}(n-1,2p_+(1-p_+/2))$. Since $2(n-1)p_+(1-p_+/2)\leq 0.798n$, the Chernoff bound~\eqref{eq:chern_lb} implies that $|N(u)\cup N(v)|\geq 4n/5$ in $\mathcal{G}(n,p_+)$ with probability $\exp(-\Omega(d))$.

Next, let $d<n/20$. By the same argument, we get that whp every pair of vertices has at most $\frac{1.2d^2}{n}$ common neighbours. Therefore,
$$
 |N(M'_{u,v})\setminus(N(\{u,v\})\cup N(M_{u,v}))|>\left\lceil\frac{n}{20d}\right\rceil\left(d-\left(\frac{n}{10d}+1\right)\frac{1.2d^2}{n}\right)>\frac{n}{25},
$$
which finishes the proof of the lemma.
\end{proof}

\begin{proof}[Proof of \cref{thm:dense-full}]
Due to the Expander Mixing Lemma and \cref{th:lambda}, whp between any set $N$ of size $\Omega(d)$ and any set $W$ of size $\Omega(n)$, there are $(1\pm o(1))|N||W|\frac{d}{n}$ edges. We denote the intersection of this event with the event from the assertion of \cref{lem:neighbourhood} by $\mathcal{E}$.

Suppose that $k$ is as large as needed, and fix any partition $[n] = V_1 \sqcup \ldots \sqcup V_k$ such that each part has size in the range $[n/3k,3n/k]$ as in the statement of the theorem.  For each $i\in [k]$, define $d_i: [n] \to \mathbb{Z}$ so that $d_i(w) = |V_i\cap N(w)|$. 
 Finally, for each vertex $u\in [n]$, we interpret $c_i(u)$ as the colour of $u$ after $i$ rounds of CR. 
 For our goal, it suffices to show that whp 
 no two vertices have the same value of $c_3(\cdot)$.

We proceed as follows. Fix a pair of vertices $u,v\in [n]$ and expose the neighbourhoods of $u$ and $v$.
Note that $c_3(u)=c_3(v)$ if and only if there exists a bijection $b:N(v)\mapsto N(u)$ such that for any $w\in N(v)$ we have $c_2(b(w))=c_2(w)$. Fix such a bijection (in $d!$ ways). Define $N':=N(v)\setminus(N(u)\cup\{u\})$.
Choose arbitrarily a set $M'\subset N'$ of $\lfloor n/(20d)\rfloor $ vertices from $N'$, and let $M=b^{-1}(M')\subset N(U)$. Expose all edges that touch $M\cup N'$. Let 
$$
M''=N'\cup N(M')\setminus(N(u)\cup N(M)\cup\{u,v\}).
$$
We extend the bijection $b$ to an injection $b:N(V)\cup N(M)\mapsto N(U)\cup N(M')$ such that, for every $w\in M$ and every $w'\in N(w)$, we get $c_1(b(w'))=c_1(w')$ and $b(w')\in N(b(w))$. The number of ways to define such an extension is at most $(d!)^{n/20d}$.

Let $W:=[n]\setminus (\{u,v\}\cup N(u)\cup N(v)\cup N(M)\cup N(M'))$, and $W_i:=W\cap V_i$. After exposing every edge except those between $M''$ and $W$, we can determine the values of $d_i(w)$ for every $w\in N(M)$ (as every neighbour of every vertex in $N(M)$ has been exposed). Therefore, the injection $b$ also identifies $d_i(w)$ for every $i\in[k]$ and every $w\in M''$. Let $S$ denote the number of pairs $(i,w)$ with $1\le i \le k$ and $w\in M''$ such that $|W_i|\ge n/(30k)$ and $|W_i|d/(2n) \le |N(w)\cap W_i| \le 3|W_i|d/(2n)$. The number of ways to choose the remaining neighbours of the vertices in $M''$ is 
$$
\prod_{i=1}^k \prod_{w\in M''} \binom{|W_i|}{d_i(w)}\le \left(\frac{960k}{d}\right)^{S/2} \binom{|W||M''|}{dn/2-m},
$$
where $m$ is the number of exposed edges.
Indeed for any positive integers $a_1,a_2,b_1,b_2$ with $b_1\ge b_2$ and $b_i \le 3a_i/4$ for $i=1,2$ we have by \cref{cor:binomial_product} that
\begin{align*}
\binom{a_1}{b_1}\binom{a_2}{b_2}& \le \frac{2}{3}\sqrt{\frac{(b_1+b_2)(a_1+a_2)a_1a_2}{(a_1+a_2)b_1(a_1-b_1)b_2(a_2-b_2)}}\binom{a_1+a_2}{b_1+b_2}\\
&\le\frac{8}{3}\sqrt{\frac{b_1+b_2}{b_1b_2}}\binom{a_1+a_2}{b_1+b_2}\le {\frac{4}{\sqrt{b_2}}}\binom{a_1+a_2}{b_1+b_2}.
\end{align*}

Let us show that the event $\mathcal{E}$ implies $S\ge kn/1100$. Indeed, this event implies that $|W|\geq n/5$ (if $d>n/20$, then $W=[n]\setminus (N(\{u,v\})\cup\{u,v\})$ and has size at least $n/5$ by the second assertion of \cref{lem:neighbourhood}; if $d\leq n/20$, then $|W|\geq n-2d-2(n/(20d))d>n/5$). Therefore, there are at least $n/6$ vertices in the union of $W_i$ such that $|W_i|\geq n/(30k)$. Thus, there are at least $(n/6)/(3n/k)=k/18$ such $W_i$. Fix such a $W_i$. Since $\mathcal{E}$ holds, any subset $\tilde N\subset M''$ of size $\Omega(n)$ sends $(1\pm o(1))|\tilde N||W_i|$ edges to $W_i$. Moreover, $|M''|\geq n/60$. Indeed, if $d>n/20$, then $M''=N'$ which has size at least $d/3>n/60$ by the first assertion of \cref{lem:neighbourhood}; if $d\leq n/20$, then $|M''|>n/25$ by the third assertion. The event $\mathcal{E}$ also implies that number of vertices in $M''$ that have less than $|W_i|\frac{d}{2n}$ or more than $|W_i|\frac{3d}{2n}$ edges in $W_i$, is $o(n)$. So, indeed $S>(1-o(1))(k/18)(n/60)>kn/1100$.

Using \eqref{eq:Gnd} and letting $g(\mathbf{d})=(d,\ldots,d)\in\mathbb{Z}^n$, the probability that there exists $u,v\in [n]$ and a partition $V_1\sqcup \ldots \sqcup V_k$ such that $c(u)=c(v)$ is then at most
\begin{multline*}
\mathbb{P}(\neg\mathcal{E})+\frac{1}{g(\mathbf{d})}n^2 2^{kn} d! \sum_{N'\subseteq [n]} \sum_{W\subseteq [n]\setminus N'} \sum_{m=0}^{dn/2} \binom{\binom{n}{2}-|N'||W|}{m}\left(\frac{960k}{d}\right)^{S/2} \binom{|W||N'|}{dn/2-m}\\
=o(1)+d^{\Theta(n)} \left(\frac{960k}{d}\right)^{kn/2200}=o(1),
\end{multline*}
when $k$ is sufficiently large. This completes the proof of \cref{thm:dense-full}.
\end{proof}

\section{Proof of \cref{thm:main_complete_refinement}: Sparse Case}
\label{sc:sparse-full}

Here we prove the following.

\begin{theorem}\label{thm:sparse-full}
There exists a universal constant $k$ such that the following holds. Let $d_0=d_0(k)$ be large enough, let $d_0\leq d\leq n^{10/17}$, and let $\mathbf{G}_n \sim \mathcal{G}_{n,d}$. 
Then whp: for every partition $[n] = V_1 \sqcup \ldots \sqcup V_k$ of the vertex set of $\mathbf{G}_n$ such that for any $i \in [k]$, $n/3k \le |V_i| \le 3n/k$, after $\diam (\mathbf{G}_n)+1$ rounds of CR, there are only singleton colour classes.
\end{theorem}

We will use two direct corollaries of \cref{lm:common_expansion_Thm3} from \cref{sc:sparse-many-large} in this proof. We first state these in \cref{sc:balls-sizes}, and then prove \cref{thm:sparse-full} in \cref{sc:sparse-full-final}.



\subsection{Sizes of Balls}
\label{sc:balls-sizes}

As in \cref{lm:common_expansion_Thm3}, let $c>0$ be small enough and $d_0$ be large enough. Let $d_0\leq d=o(n)$ and let $\mathbf{G}_n\sim\mathcal{G}_{n,d}$. The following two lemmas are direct corollaries of \cref{lm:common_expansion_Thm3}.

\begin{lemma}
\label{lm:balls_lower_bound}
Whp, for every $r$ such that $d(d-1)^{r-1}\leq c n$, the following holds
\begin{itemize}
\item for every vertex $u$, 
$$
|S_r(u)|\ge\left(1-100c-\frac{4\ln d}{d}\right)d(d-1)^{r-1};
$$
\item for every pair of vertices $u\neq v$, 
$$
|S_r(v)\setminus B_r(u)|\geq|S_r(\{u,v\})|-|B_r(u)|\geq\left(1-200c-\frac{8\ln d}{d}-\frac{1}{d-2}\right)d(d-1)^{r-1}.
$$
\end{itemize}
\end{lemma}
\begin{proof}
The first assertion is just \cref{lm:common_expansion_Thm3} applied with a singleton $U=\{u\}$. The second follows from $|U|=2$ in \cref{lm:common_expansion_Thm3} together with the basic bound $|B_r(u)|\leq \sum_{i\leq r-2}d(d-1)^i$.
\end{proof}

\begin{lemma}
\label{lm:common_expansion}
Whp
\begin{itemize}
\item for every set $U$ of size $\frac{cn}{d}$, there are at least $(1-4\ln d/d-100c)cn$ vertices that have a neighbour in $U$;
\item for any two disjoint sets $U,V$ of size $cn/d$, the number of vertices that have neighbours both in $U$ and in $V$ is at most $|N(U)|/10$.
\end{itemize}
\end{lemma}

\begin{proof}
The first assertion follows immediately from \cref{lm:common_expansion_Thm3} applied with $r=1$. The second assertion follows as well since whp for any two disjoint sets $U$ and $V$, the number of vertices that have neighbours in both sets is at most
\begin{align*}
 |N(U)|+|N(V)|-|N(U\cup V)| &\leq 2d(cn/d)-(1-4\ln d/d-100c)2cn\\
 &=(4\ln d/d+100c)2cn\\
 &\leq\frac{1}{10}(1-4\ln d/d-100c)cn\leq|N(U)|/10.
\end{align*}
This finishes the proof of the lemma. 
\end{proof}

\subsection{Colour Refinement Run on a Vertex-coloured Random Graph}
\label{sc:sparse-full-final}

Let $d$ be large enough. In what follows we assume that properties from \cref{lm:balls_lower_bound} and \cref{lm:common_expansion} hold in $\mathbf{G}_n$ deterministically.  

Fix a partition $[n]=V_1\sqcup\ldots\sqcup V_k$ as in the statement of the theorem. Assign to every vertex $x$ the colour $C_0(x)$ that equals the index of the set $V_i$ to which $x$ belongs. Let $\mathbf{D}$ be the diameter of $\mathbf{G}_n$. Consider the output $C_t$ of $t:=\mathbf{D}+1$ rounds of CR at the coloured graph. We want to prove that $C(u)\neq C(v)$ for any two different vertices $u,v\in[n]$.
 
Fix two vertices $u\neq v$. Assume $C_t(u)=C_t(v)$. Then, for every neighbour $x$ of $v$, there exists a neighbour $y$ of $u$ such that $C_{t-1}(x)=C_{t-1}(y)$. More generally, we have the following.

\begin{claim}
\label{cl:colours_descent}
Let $r\in[t]$. For every vertex $a$ and every vertex $b$ such that $C_{t-r}(a)=C_{t-r}(b)$, and every neighbour $x$ of $a$, there exists a neighbour $y$ or $b$ such that $C_{t-r-1}(x)=C_{t-r-1}(y)$.
\end{claim}

Let $r=\lfloor\log_{d-1}(\varepsilon n)\rfloor$, where $\varepsilon>0$ is a small enough constant. Due to the property from the conclusion of \cref{lm:balls_lower_bound}, we have that 
\begin{align*}
|S_r(u)\setminus B_r(v)|&\geq\left(1-\frac{4}{d}-\frac{1}{d-1}-\frac{2}{\ln n}-O\left(\varepsilon\right)\right)d(d-1)^{r-1}\\
&>\frac{d}{(d-1)^2}\varepsilon n\left(1-\frac{4}{d}-\frac{1}{d-1}-\frac{2}{\ln n}-O\left(\varepsilon\right)\right)
>\frac{\varepsilon}{2d}\cdot n.
\end{align*}

Due to \cref{cl:colours_descent}, for every vertex $x\in S_1(u)\setminus B_1(v)$, there exists a vertex $f(x)\in B_1(v)$ such that $C_{t-1}(x)=C_{t-1}(f(x))$. Next, for every vertex $x\in S_2(u)\setminus B_2(v)$, let $\pi(x)\in S_1(u)\setminus B_1(v)$ be one of its ``parents''. We have that $C_{t-1}(\pi(x))=C_{t-1}(f(\pi(x))$. Therefore, by \cref{cl:colours_descent}, there  exists $f(x)\in N(f(\pi(x)))\subset B_2(v)$ such that $C_{t-2}(x)=C_{t-2}(f(x))$. We then define $f:B_r(u)\setminus B_r(v)\to B_r(v)$ by induction: for every $2\leq i\leq r$, assuming that $f$ has been defined on $B_{i-1}(u)\setminus B_{i-1}(v)$, and for every $x\in S_i(u)\setminus B_i(v)$, find its ``parent'' $\pi(x)\in S_{i-1}(u)\setminus B_{i-1}(v)$ and take $f(x)\in N(f(\pi(x))$ such that $C_{t-i}(x)=C_{t-i}(f(x))$.


Take $U\subset S_r(u)\setminus B_r(v)$ of size $\frac{\varepsilon n}{2d}$ and let $U':=f(U)\subset B_r(v)$. We have $|U'|\leq|U|$. Without loss of generality, we assume $|U'|=|U|$ (otherwise, we can extend $U'$ arbitrarily to keep the two sets disjoint, and the argument below will still work). Note that $B_r(u)=B_{r-1}(u)\cup N(S_{r-1}(u))$, that 
$$
|S_{r-1}(u)|\leq d(d-1)^{r-2}\leq\frac{d}{(d-1)^2}\varepsilon n<1.1\varepsilon\frac{n}{d},
$$
and that $U$ and $S_{r-1}$ are disjoint. 
The same facts hold for $B_r(v)$. In particular, 
$$
 |S_{r-1}(u)\cup S_{r-1}(v)|<3\varepsilon\frac{n}{d}.
$$
Therefore, by the conclusion of \cref{lm:common_expansion}, we have that 
\begin{align*}
|N(U)\setminus (N(U')\cup B_r(u)\cup B_r(v))| &>|N(U)|-\frac{4}{10}|N(U)|-|B_{r-1}(u)|-|B_{r-1}(v)|\\
&>\frac{6}{10}(1-(\ln 40)/10-\Theta(\varepsilon))\frac{1}{2}\varepsilon n-2\frac{d}{d-2}(d-1)^{r-1}\\
&>0.18\cdot \varepsilon n-
2\frac{d\varepsilon n}{(d-2)(d-1)}>0.1\cdot \varepsilon n.
\end{align*}
Let $\mathcal{N}$ be a subset of $N(U)\setminus (N(U')\cup B_r(u)\cup B_r(v))$ of size $\varepsilon n/10$.

We then extend $f$ to $\mathcal{N}$: Each vertex $x\in \mathcal{N}$ has $f(x)\in N(U')$. Note that the set 
$$
X:=[n]\setminus(B_r(u)\cup B_r(v)\cup\mathcal{N}\cup f(\mathcal{N}))
$$ 
has size at least
$$
 n-2d(d-1)^{r-1}-\frac{\varepsilon n}{2d}\cdot(2d)>n-2\frac{d}{d-1}\varepsilon\cdot n-\varepsilon\cdot n\geq n(1-4\varepsilon). 
$$
The set $X$ is partitioned into $k$ sets $X=V'_1\sqcup\ldots\sqcup V'_k$ so that $n(\frac{1}{3k}-\varepsilon)\leq |V'_i|\leq n\cdot\frac{3}{k}$. 

Due to \cref{cl:regular-probability-bound}, whp any set of size at most $\varepsilon n/3$ induces at most $\frac{1}{100}\varepsilon dn$ edges: 
$$
 {n\choose \varepsilon n/3}{\varepsilon^2 n^2/18\choose \frac{1}{100}\varepsilon dn}\left(\frac{2d}{n}\right)^{\frac{1}{100}\varepsilon dn}\leq\left(\left(\frac{3e}{\varepsilon}\right)^{100/3}\left(100e\varepsilon/9\right)^d\right)^{\frac{1}{100}\varepsilon n}=o(1)
$$
since $d$ is large and $\varepsilon$ is small enough. In particular, we may assume that there there are at most $\frac{1}{100}\varepsilon dn$ edges between $\mathcal{N}$ and $\mathcal{N}\cup f(\mathcal{N})\cup U$. Since $\mathcal{N}$ does not have any other neighbours in $B_r(u)\cup B_r(v)$, we get that there exists a subset $\mathcal{N}_0\subset\mathcal{N}$ of size $\frac{\varepsilon n}{50}$ such that each vertex in this set sends at least $\frac{3}{4}d$ edges to $X$.

Note that, for any vertex $x\in\mathcal{N}_0$, the equality $C_{t-r-1}(x)=C_{t-r-1}(f(x))$ implies $C_1(x)=C_1(f(x))$. Therefore, as soon as the sets $B_r(u),B_r(v)$ are exposed, the set $U$ is chosen, the sets $N(U),N(U'),N(N(U'))$ are exposed, and the set $\mathcal{N}_0$ is chosen, there should exist a function $f$ defined as above, that identifies the values of $|N_{V'_j}(x)|$ for every $x\in\mathcal{N}_0$ and $j\in[k]$.

Therefore, we run the following exploration process of the random graph. First, we expose $B_r(u),B_r(v)$ and then choose $U\subset S_r(u)\setminus B_r(v)$ of size $\frac{\varepsilon n}{2d}$ arbitrarily. We then expose $N(U)$, $N(U')$, and $N(N(U'))$. We choose $f$ on $(B_r(u)\setminus B_r(v))\cup\mathcal{N}$ in at most $d^{2\varepsilon n}$ ways, since 
$$
 |(B_r(u)\setminus B_r(v))\cup\mathcal{N}|\leq |B_r(u)|+|N(U)|\leq \frac{d}{d-2}(d-1)^{r}+\frac{\varepsilon n}{2}\leq \frac{d}{d-2}\varepsilon n+\frac{\varepsilon n}{2}<2\varepsilon n.
$$
Finally, we choose any set $\mathcal{N}_0\subset\mathcal{N}$ of size $\frac{\varepsilon n}{50}$ such that each vertex in this set sends at least $\frac{3}{4}d$ edges to $X$.

By the Expander Mixing Lemma and \cref{th:lambda}, whp between any two disjoint sets of size at least $n/(4k)$ and $n/2$, there are at least $dn/(10k)$ edges, and every set of size at least $n/2$ induces at least $dn/10$ edges. 

Recall that every $x\in\mathcal{N}_0$ has a prescribed number of neighbours $g_j(x)$ in the set $V'_j$.
By the Expander Mixing Lemma and \cref{th:lambda}, whp the number of edges between any two disjoint sets $U$ and $V$ of sizes $\Theta(n)$ equals $|U||V|d(1\pm\varepsilon)/n$. Therefore, for every set $V'_j$, there exists a subset $\mathcal{N}'_j\subset\mathcal{N}_0$ of size $\varepsilon n/100$ such that every $x\in\mathcal{N}'_j$ has 
$$
g_j(x)\in[d/(10k),10d/k].
$$ 

Let us estimate the probability that for every $j\in[k]$, every vertex from $\mathcal{N}'_j$ has $g_j(x)$ neighbours in $V'_j$. For every $j$, we order arbitrarily the vertices in $\mathcal{N}'_j$: $x^j_1,\ldots,x^j_t$, where
$$
 t=\varepsilon n/100.
$$
Let $\mathcal{E}$ be the event that, for every $j\in[k]$, every $x^j_i$ has $g_j(x^j_i)$ neighbours in $V'_j$. Let $h^j_{in}\geq dn/10$ and $h^j_{out}\geq dn/(10k)$ be integers such that 
$$
\mathbb{P}\left(\mathcal{E}\wedge \bigwedge_{j=1}^k \left\{|E(\mathbf{G}_n[X\setminus V'_j]|=h^j_{in},\,|E(\mathbf{G}_n[V'_j\times(X\setminus V'_j)]|=h^j_{out}\right\}\right)\text{ is maximum}.
$$ 
Let $\Sigma_0$ be the set of all $d$-regular graphs $G$ on $[n]$ satisfying $\mathcal{E}$ and such that, for all $j\in[k]$, $G[X\setminus V'_j]$ and $G[V'_j\times (X\setminus V'_j)]$ have exactly $h^j_{in}$ and $h^j_{out}$ edges, respectively\footnote{Sets $X$ and $V'_j$ depend on $G$: given a graph $G$, we expose the balls around $u$ and $v$, which identify these sets. In what follows, we will perform switching operations on $G$ that preserve the exposed balls and, therefore, sets $X$ and~$V'_j$.}. The following claim completes the proof of \cref{lm:from_single_set-to_partition-gen}.

\begin{claim}
$\mathbb{P}(\mathbf{G}_n\in\Sigma_0)\leq (k/d)^{\varepsilon nk/500}$.
\label{cl:large-switchings-2}
\end{claim}

Indeed, \cref{cl:large-switchings-2} implies that $\mathbb{P}(\mathcal{E})\leq (dn)^{2k}(k/d)^{\varepsilon nk/500}$. Therefore, by the union bound  
$$
 \mathbb{P}(C_t(u)=C_t(v))\leq d^{2\varepsilon n}(dn)^{2k}(k/d)^{\varepsilon nk/500}=e^{-\Omega(kn)}
$$
when $k$ is large enough and $d\gg k$. The union bound over the choice of partition $V_1\sqcup\ldots\sqcup V_k$ and over all pairs of distinct vertices $u,v$ completes the proof of \cref{thm:sparse-full}.

The proof of \cref{cl:large-switchings-2} can basically be obtained by copying the proof of \cref{cl:large-switchings} verbatim. Nevertheless, we present the proof below in nearly complete detail.

\begin{proof}[Proof of \cref{cl:large-switchings-2}.]

Assume the opposite:  $\mathbb{P}(\mathbf{G}_n\in\Sigma_0)>(k/d)^{\varepsilon nk/500}$. Let $\mathcal{X}$ be the set of $d$-regular graphs $G$ on $[n]$ such that there exists a vertex $x$ and a set $V$ of size $n/(3k)$ with at least $n(d/k)^{0.3}$
edges between $N(x)$ and $V$. Then, due to \cref{cl:regular-probability-bound},
\begin{align}
 \mathbb{P}(\mathbf{G}_n\in\mathcal{X})&\leq n{n\choose n/(3k)}{3dn/k\choose n(d/k)^{0.3}}\left(\frac{2d}{n}\right)^{n(d/k)^{0.3}}\notag\\
 &\leq n{n\choose n/(3k)}\left(\frac{3e(d/k)^{0.7}\cdot 2d}{n}\right)^{n(d/k)^{0.3}}=o\left(e^{-n(d/k)^{0.3}}\right),
 \label{eq:X-small-2}
\end{align}
since $d\leq n^{10/17}$ and $k$ is large enough.

In the same way as in the proof of \cref{cl:large-switchings}, we define sets $\Sigma^j_{\ell}$, $j\in[k]$, $\ell\in[t]$: for $|\delta|\leq(d/k)^{0.3}$, we define $\Sigma^j_{\ell}(\delta)$
as the set of $d$-regular graphs $G$ on $[n]$ such that
\begin{itemize}
\item for every $i\geq j+1$, each vertex $x\in \mathcal{N}'_i$ has exactly $g_i(x)$ neighbours in $V'_i$,
\item for every $\ell+1\leq i\leq t$, the vertex $x_i^j$ has exactly $g_j(x_i^j)$  neighbours in $V'_j$,
\item $x_{\ell}^j$ has $g_j(x_{\ell}^j)+\delta$ neighbours in $V'_j$,
\item for every $i\geq j-1$, the number of neighbours of every vertex $x\in \mathcal{N}'_i$ in $V'_i$ satisfies $|N_{V'_i}(x)-g_i(x)|\leq(d/k)^{0.3}$;
\item for every $i\in[\ell-1]$, $|N_{V'_j}(x^j_i)-g_j(x^j_i)|\leq (d/k)^{0.3}$,
\item for every $i\in[k]$, the number of edges in $G[X\setminus V'_i]$ and $G[V'_i\times (X\setminus V'_i)]$, denoted by $h_{in}^i(j,\ell)$ and $h_{out}^i(j,\ell)$, respectively, satisfy:
$$
 \left|h_{in}^i(\ell,j)-h_{in}^i\right|\leq(d/k)^{0.3}((j-1)t+\ell),\quad
 \left|h_{out}^i(\ell,j)-h_{out}^i\right|\leq(d/k)^{0.3}((j-1)t+\ell).
$$ 
\end{itemize}
Set $\Sigma_{\ell}^j:=\bigcup_{|\delta|\leq(d/k)^{0.3}}\Sigma_{\ell}^j(\delta)$. The crucial observation is that, at every step, the bound on the difference of the number of edges in $X\setminus V'_i$ and $V'_i\times (X\setminus V'_i)$ with $h^i_{in}$ and $h^i_{out}$ is at most
$$
 (d/k)^{0.3}kt=O\left(\min_j\{h^j_{out},h^j_{in}\}\cdot k^{1.7}/d^{0.7}\right).
$$
In order to get a contradiction with our initial assumption, we prove that 
\begin{itemize}
    \item $|\Sigma_0|\leq(d/k)^{-0.2}|\Sigma^1_1|$,
    \item for every $j\in[k]$ and $\ell\in[t-1]$, $|\Sigma^j_{\ell}|\leq (d/k)^{-0.2}|\Sigma^j_{\ell+1}|$,
    \item for every $j\in[k-1]$, $|\Sigma^j_t|\leq (d/k)^{-0.2}|\Sigma^{j+1}_1|$.
\end{itemize} 
We prove this using switchings. For instance, in order to prove $|\Sigma^j_{\ell}|\leq (d/k)^{-0.2}|\Sigma^j_{\ell+1}|$, we take $G\in\Sigma^j_{\ell+1}(\delta)$ and consider a tuple of vertices $(y,u,v)$ such that 
\begin{itemize}
\item $y\notin V'_j$, $u\in V'_j$, $v\notin V'_j$,
\item and $\{x^j_{\ell+1},y\},\{u,v\}\in E(G)$, $\{x^j_{\ell+1},u\},\{y,v\}\notin E(G)$.
\end{itemize}
If we switch
\begin{equation}
 \{x^j_{\ell+1},y\},\{u,v\}\mapsto\{x^j_{\ell+1},u\},\{y,v\},
\label{eq:switch-def-3}
\end{equation}
we get a graph from $\Sigma^j_{\ell+1}(\delta+1)$. It gives
\begin{multline*}
 |\Sigma_{\ell+1}^j(\delta)|\cdot(|N_X(x^j_{\ell+1})|-(g_j(x^j_{\ell+1})+\delta))\cdot(h^j_{out}\pm O((d/k)^{0.3}kt))-|\mathcal{X}\cap\Sigma_{\ell+1}^j(\delta)|O(d^2n)\\
 =|\Sigma_{\ell+1}^j(\delta+1)|\cdot(g_j(x^j_{\ell+1})+\delta+1)\cdot(h^j_{in}\pm O((d/k)^{0.3}kt+d^2)).
\end{multline*}
In a similar way, we apply the switching operation~\eqref{eq:switch-def-3} to $G\in\Sigma_{\ell+1}^j(\delta)$ and get a graph from $\Sigma_{\ell+1}^j(\delta-1)$, where 
\begin{itemize}
\item $y\in V'_j$, $u,v\notin V'_j$,
\item and $\{x_{\ell+1}^j,y\},\{u,v\}\in E(G)$, $\{x_{\ell+1}^j,u\},\{y,v\}\notin E(G)$.
\end{itemize}
We get
\begin{multline*}
 |\Sigma_{\ell+1}^j(\delta)|\cdot(g_j(x^j_{\ell+1})+\delta)\cdot(h^j_{in}\pm O((d/k)^{0.3}kt+d^2))\\
 =|\Sigma_{\ell+1}^j(\delta-1)|\cdot(|N_X(x^j_{\ell+1})|-(g_j(x^j_{\ell+1})+\delta-1))\cdot(h^j_{out}\pm O((d/k)^{0.3}kt))-|\mathcal{X}\cap\Sigma_{\ell+1}^j(\delta-1)|O(d^2n).
\end{multline*}
If $(|N_X(x^j_{\ell+1})|-g_j(x^j_{\ell+1}))h^j_{out}>g_j(x^j_{\ell+1}) h^j_{in}$, then $\frac{|\Sigma_{\ell+1}^j(\delta+1)|}{|\Sigma_{\ell+1}^j(\delta)|}>1-O(k^{1.7}/d^{0.7})$, implying 
\begin{equation}
|\Sigma^j_{\ell+1}|/|\Sigma^j_{\ell}|\geq(1-O(k^{1.7}/d^{0.7}))^{(d/k)^{0.3}}(d/k)^{0.3}>(d/k)^{0.2}.
\label{eq:final-sigmas-blowup}
\end{equation}
Otherwise, $\frac{|\Sigma_{\ell+1}^j(\delta-1)|}{|\Sigma_{\ell+1}^j(\delta)|}>1-O(k^{1.7}/d^{0.7})$, implying~\eqref{eq:final-sigmas-blowup}, as well. The proof of \cref{cl:large-switchings-2} is now complete.
\end{proof}

\section{Proof of~\cref{lm:common_expansion_Thm3}}
\label{sc:expansion-lemma-proof}

Let us fix $U\subseteq [n]$ of size  $u = |U| \leq\frac{cn}{d}$. (We may assume that $d\leq cn$ as the statement is vacuously true otherwise.) 
Let $c\geq c':=ud/n$. For $i\geq 1$, denote 
$$
\varepsilon_i:=\frac{2\ln d}{2^{i-1}\cdot d}+50c'(d-1)^{i-1}.
$$
We start from the following claim.

\begin{claim}
For every positive integer $r$ such that $ud(d-1)^{r-1} \le cn$,
$$
 \mathbb{P}\biggl(|S_r(U)|\geq(1-3\ln d/d-100c)ud(d-1)^{r-1}\biggr)\geq 1-\sum_{i=1}^r\left(\frac{2ec'(d-1)^{i-1}}{\varepsilon_i}\right)^{\varepsilon_i d(d-1)^{i-1}u}.
$$
\end{claim}

\begin{proof}
Fix a positive integer $r$ such that $d(d-1)^{r-1}u \le cn$. Assuming that all the edges in $B_{r-1}(U)\setminus E(\mathbf{G}_n[S_{r-1}(U)])$ have been explored, let as explore the next layer $S_r(U)$. For $i\leq r$, let $S^1_i(U)$ be the set of vertices in $S_i(U)$ that have degree 1 in $B_i(U)\setminus E(\mathbf{G}_n[S_i(U)])$. 

We assume, by induction, that
$$
 |S^1_{r-1}(U)|\geq d(d-1)^{r-2}u\left(1-2\sum_{i=1}^{r-1}\varepsilon_i\right)
$$ 
whenever $r\geq 2$ (if $r=1$, we do not impose any assumption). Enumerate the vertices so that $S^1_{r-1}(U)=\{x_1,x_2,\ldots\}$. We order arbitrarily the edges that we have to explore from $x_j$'s outside of $B_{r-1}(U)$. Assume that we have exposed all adjacencies for $x_1,\ldots,x_{j-1}$, and the first $\ell-1$ edges of $x_j$. Let $S' \subseteq S_r(U)$ be the set of neighbours of $\{x_1,\ldots,x_j\}$ outside of $B_{r-1}(U)$ that have been explored so far. 


Let us show that the $\ell$-th edge $\{x_j,v\}$ touches $S'\cup S_{r-1}(U)$ with probability at most $2c'(d-1)^{r-1}$, which can be done using switchings. Consider the family of $d$-regular graphs on the vertex set $[n]$ that contain the explored edges. We partition this family into two subfamilies, $\Sigma_0$ and $\Sigma_1$. 
 In $\Sigma_0$ the $\ell$-th edge $\{x_j, v\}$ {\it does not} touch $S'\cup S_{r-1}(U)$ whereas in $\Sigma_1$ it does. For every $G\in\Sigma_1$ (with the $\ell$-th edge $\{x_j,v\}$ identified) and any pair of vertices $(z,w)$ such that
 \begin{itemize}
 \item $\{z,w\}\in E(G)$, $\{x_j,z\},\{v,w\}\notin E(G)$,
 \item $z,w\notin S'\cup S_{r-1}(U)$,
 \end{itemize}
 we may switch 
 $$
 \{x_j,v\},\{z,w\}\mapsto\{x_j,z\},\{v,w\}
 $$ 
 to get a graph from $\Sigma_0$ (with the $\ell$-th edge $\{x_j,z\}$ identified). To compare the sizes of $\Sigma_1$ and $\Sigma_0$, we need to estimate the number of triples $(G, z, w)$, $G \in \Sigma_1$, and the number of triples $(G, v, w)$, $G \in \Sigma_0$. For a given $G \in \Sigma_1$, the number of pairs $(z,w)$ is at least 
\begin{align*}
 d(n-|B_{r-1}(U)|-|S'|)&-d(|B_{r-1}(U)|+|S'|)-2d^2=d(n-2|B_{r-1}(U)|-2|S'|-2d)\\
 &\geq d\left(n-2u\frac{d}{d-2}(d-1)^{r-1}-2ud(d-1)^{r-1}-2d\right).
\end{align*}
On the other hand, for a given $G \in \Sigma_0$, the number of pairs $(v,w)$ is clearly at most 
$$
(|S_{r-1}(U)|+|S'|)d\leq d(ud(d-1)^{r-2}+ud(d-1)^{r-1}).
$$ 
Thus, we get 
$$
|\Sigma_1| (n-2ud(d-1)^{r-1}-2ud(d-1)^{r-1}/(d-2)-2d) \leq |\Sigma_0| (ud(d-1)^{r-2}+ud(d-1)^{r-1}),
$$
implying the desired probability bound
\begin{align*}
 \frac{|\Sigma_1|}{|\Sigma_1|+|\Sigma_0|}\leq\frac{|\Sigma_1|}{|\Sigma_0|}&\leq\frac{ud(d-1)^{r-2}+ud(d-1)^{r-1}}{n-2ud(d-1)^{r-1}-2ud(d-1)^{r-1}/(d-2)-2d}\\
 &=\frac{ud(d-1)^{r-1}(1+1/(d-1))}{n(1-2d(d-1)^{r}u/(n(d-2))-2d/n)}\\
 &\leq\left(1+\frac{1}{d-1}+3\frac{d(d-1)^{r-1} u}{n}\right)\frac{d(d-1)^{r-1} u}{n(1-2c)}<2c'(d-1)^{r-1},
\end{align*}
as needed.

Therefore, the probability that the number of edges explored in this way that ``do not lead to a new vertex'' (i.e.\ do not fall into $S^1_r(U)$) is at least $\varepsilon_r d(d-1)^{r-1}|U|$ is dominated by
\begin{align*}
\mathbb{P}\biggl(\mathrm{Bin}\left[d(d-1)^{r-1}u,2c'(d-1)^{r-1}\right]&>\varepsilon_r d(d-1)^{r-1}u\biggr)\\
 &\leq{d(d-1)^{r-1}u\choose \varepsilon_r d(d-1)^{r-1}u}(2c'(d-1)^{r-1})^{\varepsilon_r d(d-1)^{r-1}u}\\
 &\leq
 \left(\frac{2ec'(d-1)^{r-1}}{\varepsilon_r}\right)^{\varepsilon_r d(d-1)^{r-1}u}.
\end{align*}
If the second endpoint of an explored edge lands in $S^1_{r-1}(U)$ as well, it immediately reveals some edge that have not been explored yet. Therefore, the above event implies that
\begin{align*}
|S_r^1(U)|&\geq (d-1)\left(d(d-1)^{r-2}u\left(1-2\sum_{i=1}^{r-1}\varepsilon_i\right)\right)-2\varepsilon_r d(d-1)^{r-1}u\\
 &=d(d-1)^{r-1}u\left(1-2\sum_{i=1}^r\varepsilon_i\right)\geq \left(1-\frac{4\ln d}{d}-100c'(d-1)^{r-1}\right)ud(d-1)^{r-1}.
\end{align*}
By induction, we get that the latter event holds with probability at least
$$
1-\sum_{i=1}^r\left(\frac{2ec'(d-1)^{i-1}}{\varepsilon_i}\right)^{\varepsilon_i d(d-1)^{i-1}u},
$$
completing the proof of the claim.
\end{proof}

In order to complete the proof of the lemma, it suffices to apply the union bound over all $u\leq\frac{cn}{d}$, all $r$ such that $ud(d-1)^{r-1}<cn$, and all $U\subseteq [n]$ of size $u$:
\begin{multline*}
 \sum_{u<cn/d}\sum_r{n\choose u}\sum_{i=1}^{r}\left(\frac{2ec'(d-1)^{i-1}}{\varepsilon_i}\right)^{\varepsilon_i d(d-1)^{i-1}u}\\
 \leq
 \sum_{u<cn/d}(\ln n)^2\max_{ud(d-1)^{r-1}<cn}\left(\frac{ed}{c'}\left(\frac{2ec'(d-1)^{r-1}}{\varepsilon_r}\right)^{\varepsilon_r d(d-1)^{r-1}}\right)^u.
\end{multline*}
If $\frac{1}{c'}\geq 25d(d-1)^{2r-2}$, then 
\begin{align}
 \ln&\left(\frac{ed}{c'}\right)-\varepsilon_r d(d-1)^{r-1}\ln\left(\frac{\varepsilon_r}{2ec'(d-1)^{r-1}}\right)\notag\\
 &\leq \ln\left(\frac{ed}{c'}\right)-2\ln d\left(\frac{d-1}{2}\right)^{r-1}\ln\left(\frac{\ln d}{de c'(d-1)^{r-1}2^{r-1}}\right)\notag\\
 &\leq \ln\left(25ed^2(d-1)^{2r-2}\right)-2\ln d\left(\frac{d-1}{2}\right)^{r-1}\ln\left(\frac{25\ln d(d-1)^{r-1}}{e 2^{r-1}}\right)<-\ln d\cdot\left(\frac{d-1}{2}\right)^{r-1},
 \label{eq:aux-bound-UB}
\end{align}
for $d$ large enough,
since $\ln(1/c)-2\ln d((d-1)/2)^{r-1}\ln(1/c)$ decreases in $1/c$.

If $\frac{1}{c'}< 25d(d-1)^{2r-2}$, then
\begin{align*}
 \ln\left(\frac{ed}{c'}\right)&-\varepsilon_r d(d-1)^{r-1}\ln\left(\frac{\varepsilon_r}{2ec'(d-1)^{r-1}}\right)\\ 
 &\leq \ln\left(\frac{ed}{c'}\right)-2\ln d\left(\frac{d-1}{2}\right)^{r-1} \ln\left(\frac{50}{2e}\right)\\
 &\leq \ln(25ed^2(d-1)^{2r-2})-2\ln d\left(\frac{d-1}{2}\right)^{r-1} \ln\left(\frac{50}{2e}\right)<-\ln d\cdot\left(\frac{d-1}{2}\right)^{r-1},
\end{align*}
as well.

We also observe that, if $\frac{1}{c'}\geq\frac{n}{d\ln n}\geq 25d(d-1)^{2r-2}$, then either $d=n^{\Theta(1)}$, or the bound~\eqref{eq:aux-bound-UB} can be improved to $-\Omega(\ln d\cdot \ln n)$.
Moreover, if $\frac{n}{d\ln n}< 25d(d-1)^{2r-2}$, then either $r=1$ and $d=n^{\Theta(1)}$, or $\ln d\cdot\left(\frac{d-1}{2}\right)=n^{\Theta(1)}$. We also note that, if $d=n^{\Theta(1)}$, then the maximum $r$ such that $ud(d-1)^{r-1}<cn/d$ is bounded.

We conclude that
\begin{align*} 
\mathbb{P}\biggl(\exists U\exists r \,\,  &|S_r(U)| <(1-50c-o_d(1))ud(d-1)^{r-1}\biggr)\\
&\leq\sum_{\ln n<u<cn/d}\ln^2 n\cdot  e^{-\ln d\cdot u}+\sum_{u\leq\ln n}\left(O(1)\cdot e^{-\Theta(\ln n\cdot u)}+e^{-(n^{-\Theta(1)}+\Omega(\ln d\ln n))\cdot u}\right)\\
&\leq \ln^2n\cdot e^{-10\ln n}+n^{-\Theta(1)}+e^{-10\ln n}=o(1),
\end{align*}
completing the proof.

\section{Acknowledgements.} The authors would like to thank the organisers of the MATRIX workshop ``Combinatorics of McKay and Wormald'', where this work originated. The last author is grateful to Oleg Verbitsky for introducing him to the topics of canonical labelling and colour refinement, and for bringing this question to his attention.

\end{document}